\documentclass[11pt,a4paper]{amsart}

\usepackage{a4wide}
\usepackage{amsfonts,amsthm,amsmath,amssymb}
\usepackage{amscd,color}
\usepackage{psfrag,graphicx,subfigure}
\usepackage{import} %% Figures with labels
\usepackage[shortlabels]{enumitem}

\theoremstyle{plain}
\newtheorem{lem}{Lemma}[section]
\newtheorem{prop}[lem]{Proposition}
\newtheorem{thm}[lem]{Theorem}

\newtheorem*{thm*}{Main Theorem}
\newtheorem*{thmA}{Theorem A}
\newtheorem*{thmB}{Theorem B}
\newtheorem*{corC}{Corollary C}
\newtheorem*{thmD}{Theorem D}

\newtheorem*{thmE}{Theorem E}

\newtheorem*{cor*}{Corollary}

\newtheorem*{cortheorem}{Correspondence Theorem}

\theoremstyle{definition}
\newtheorem{defn}[lem]{Definition}
\newtheorem*{defn*}{Definition}
\newtheorem{ex}[lem]{Example}
\newtheorem*{ex*}{Example}

\newtheorem*{rem*}{Remark}

\theoremstyle{remark}
\DeclareMathOperator{\dist}{dist}
\DeclareMathOperator{\Arg}{Arg}
\DeclareMathOperator{\Sing}{Sing}
\DeclareMathOperator{\Acc}{Acc}

\DeclareMathOperator{\IA}{IA}
\newcommand{\C}{\mathbb C}
\newcommand{\D}{\mathbb D}
\newcommand{\clC}{\widehat{\C}}
\newcommand{\chat}{\widehat{\C}}

\newcommand{\R}{\mathbb R}
\newcommand{\Z}{\mathbb Z}

\newcommand{\bd}{\partial}
\renewcommand{\Re}{\textup{Re}}
\renewcommand{\Im}{\textup{Im}}
\renewcommand{\H}{\mathbb H}

\newcommand{\inv}{^{-1}}

\newcommand{\cala}{\mathcal{A}}

\newcommand{\bdd}{\partial \D}
\newcommand{\A}{\mathcal A}
\begin{document}

\title{Accesses to infinity from Fatou components}
\date{\today}

\author{Krzysztof Bara\'nski}
\address{Institute of Mathematics, University of Warsaw,
ul.~Banacha~2, 02-097 Warszawa, Poland}
\email{baranski@mimuw.edu.pl}

\author{N\'uria Fagella}
\address{Departament de Matem\`atica Aplicada i An\`alisi,
Universitat de Barcelona, 08007 Barcelona, Catalonia, Spain}
\email{fagella@maia.ub.es}

\author{Xavier Jarque}
\address{Departament de Matem\`atica Aplicada i An\`alisi,
Universitat de Barcelona, 08007 Barcelona, Catalonia, Spain}
\email{xavier.jarque@ub.edu}

\author{Bogus{\l}awa Karpi\'nska}
\address{Faculty of Mathematics and Information Science, Warsaw
University of Technology, ul.~Ko\-szy\-ko\-wa~75, 00-661 Warszawa, Poland}
\email{bkarpin@mini.pw.edu.pl}

\thanks{Supported by the Polish NCN grant decision DEC-2012/06/M/ST1/00168. The 
second and third authors were partially supported by the Catalan grant 
2009SGR-792, and by the Spanish grant MTM2011-26995-C02-02.}

\bibliographystyle{plain}

\begin{abstract} 
We study the boundary behaviour of a meromorphic map $f: \C \to \clC$ on its invariant simply connected Fatou component $U$. To this aim, we develop the theory of accesses to boundary points of $U$ and their relation to the dynamics of $f$. In particular, we establish a correspondence between invariant accesses from $U$ to infinity or weakly repelling points of $f$ and boundary fixed points of the associated inner function on the unit disc. We apply our results to describe the accesses to infinity from invariant Fatou components of the Newton maps. 
\end{abstract}

\maketitle

\section{Introduction}\label{sec:intro}
We consider dynamical systems generated by the iteration of meromorphic maps $f:\C \to \chat$. We are especially interested in the case when $f$ is transcendental or, equivalently, when the point at infinity is an essential singularity of $f$ and hence the map is non-rational. There is a natural dynamical partition of the complex sphere $\clC$ into two completely invariant sets: the {\em Fatou set} $\mathcal{F}(f)$, consisting of points for which the iterates $\{f^n\}_{n\ge 0}$ are defined and form a normal family in some neighbourhood; and its complement, the {\em Julia set} $\mathcal{J}(f)$, where chaotic dynamics occurs. Note that in the transcendental case, we always have $\infty \in \mathcal{J}(f)$. The Fatou set is open and it is divided into connected components called {\em Fatou components}, which map among themselves. Periodic components are classified into {\em basins of attraction} of attracting or parabolic cycles, {\em rotation domains} (Siegel discs or Herman rings, depending on their genus, where the dynamics behaves 
like an irrational rotation), or {\em Baker domains} (components for which  $f^{pn}$ 
converge to infinity as $n \to \infty$ uniformly on compact sets, for some period $p\geq 1$). Components which are not eventually periodic are called {\em wandering} and they may or may not converge to infinity under iteration.

In this paper we are interested in the interplay between the dynamics of $f$ on a simply connected invariant Fatou component $U$, the geometry of the boundary of $U$ and the boundary behaviour of a Riemann map $\varphi:\D\to U$. The main motivation is to understand the structure of the Julia set near infinity of a {\em Newton map}, i.e.~a meromorphic map $N$ which is Newton's method of finding zeroes of an entire function $F$, 
\[                                                                                                                                                                                    N(z) = z - \frac{F(z)}{F'(z)}.                                                                                                                                                                                      \]
It is known that for the Newton maps, all connected components of the Fatou set are simply connected (which is not always the case for general meromorphic maps). This was proved by Shishikura \cite{shishikura} for rational Newton maps (Newton's method of polynomials) and by the authors \cite{bfjk} in the transcendental case (Newton's method of transcendental entire  maps). For a rational Newton map $N$, Hubbard, Schleicher and Sutherland studied  in \cite{hubbardschleicher} \emph{accesses} to infinity (called also \emph{channels}) from the attracting basins of $N$. They showed that the number of accesses to infinity in such a basin $U$, is finite and equal to the number of critical points of $N$ in $U$, counted with multiplicity. By describing the distribution of accesses near infinity, they were able to apply their results to find a good set of initial conditions (from a numerical point of view) for detecting a root of a polynomial of given degree. 

In the transcendental case, a Newton map on an invariant Fatou component may have infinite degree and 
its boundary behaviour is often much more complicated. Attempts to generalize the above results were done in \cite{johannes} by considering regions of the plane where the degree could be considered finite.  
However, the strict relation between the number of accesses to infinity from an invariant component $U$ of a Newton map $N$ and  its degree on $U$ does not hold in the transcendental case, as shown by the map $N(z) = z + e^{-z}$ (Newton's method applied to $F(z) = e^{-e^z}$), studied by Baker and Dom\'{\i}nguez in \cite{baker-dominguez}. Indeed, $N$ has infinitely many invariant Baker domains $U_k$, $k \in \Z$, such that $U_k = U_0 + 2k\pi i$,  $\deg N|_{U_k} = 2$ and $U_k$ has infinitely many accesses to infinity. See Figure~\ref{fig:bd}.

\begin{figure}[htbp!] \label{fig:bd}
\centering
\includegraphics[width=0.383\textwidth]{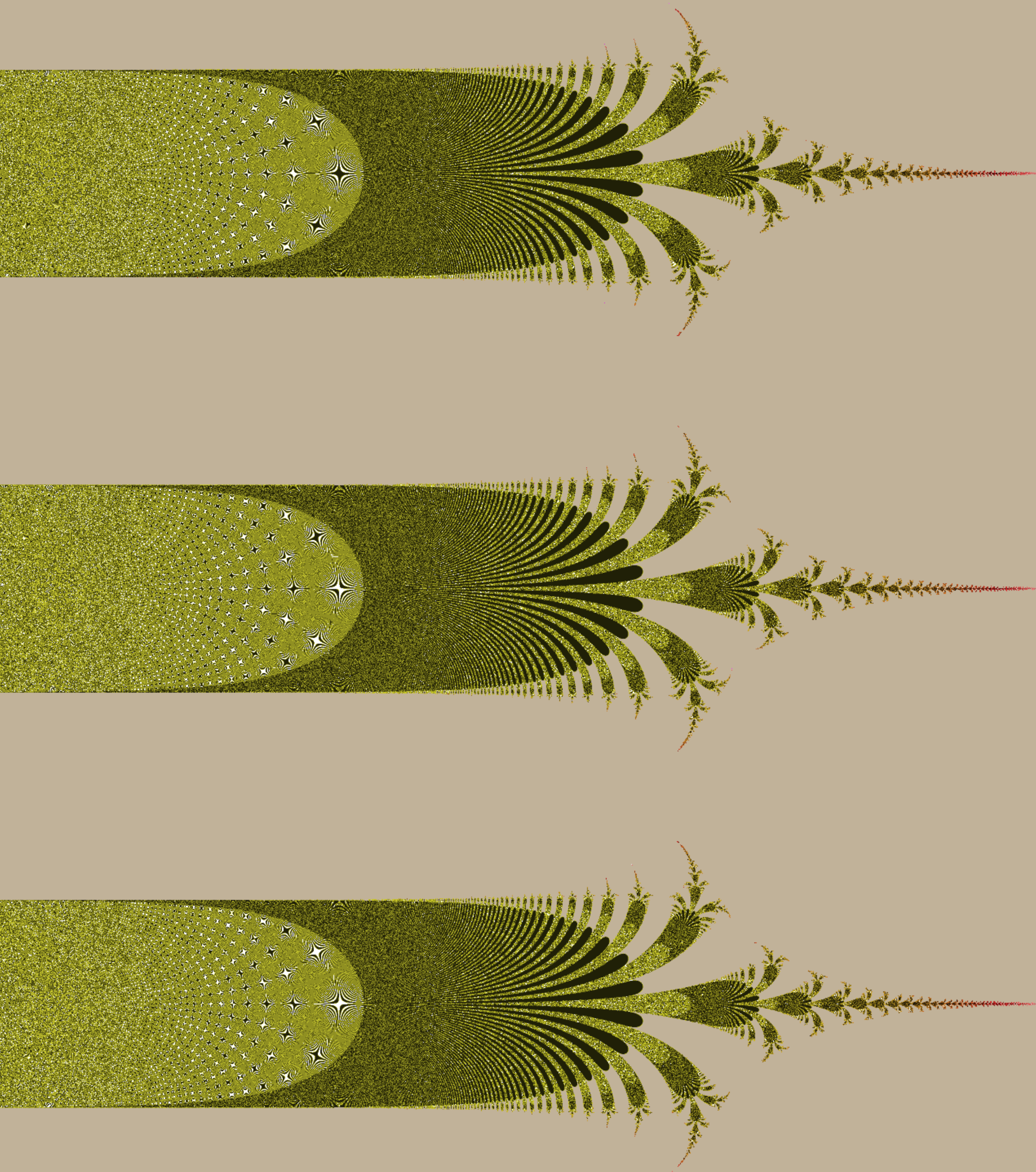}
\hfill
\includegraphics[width=0.577\textwidth]{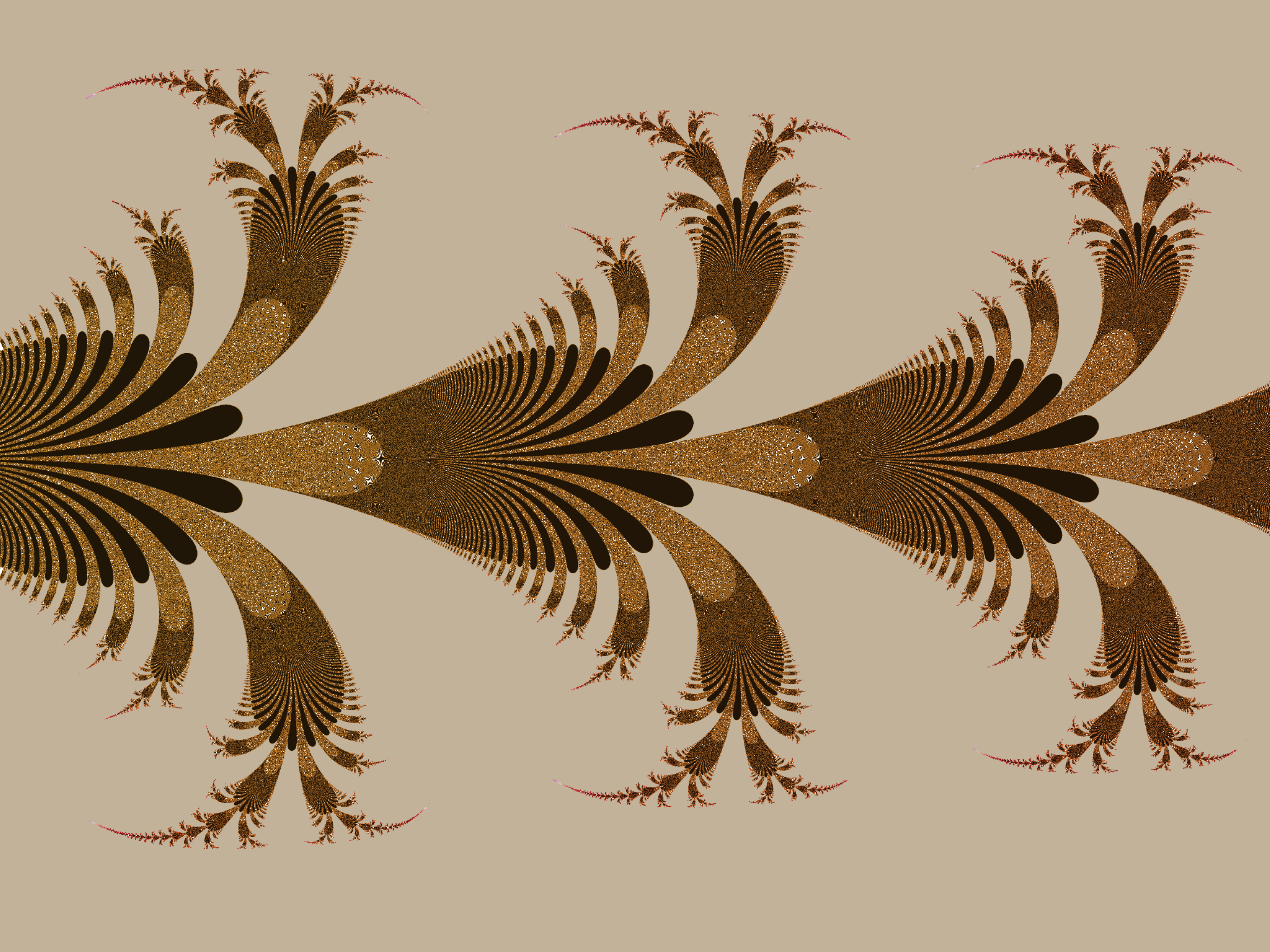}

\caption{\small Left: The dynamical plane of the map $N(z) = z +e^{-z}$, Newton's method of $F(z)=e^{-e^z}$, showing the accesses to infinity from invariant Baker domains $U_k$. Right: Zoom of the dynamical plane.}
\end{figure}

Another example of an interesting boundary behaviour is given by the Newton method applied to $F(z) = 1+z e^z$. Computer pictures suggest that the Newton map has infinitely many superattracting basins of the zeroes of $F$ on which the degree is $2$, and all but two of them show infinitely many accesses to infinity. The two special ones, adjacent to the real axis, exhibit only one access to infinity. See Figure~\ref{brown}. 

\begin{figure}[htbp!]
\centering
\includegraphics[width=0.48\textwidth]{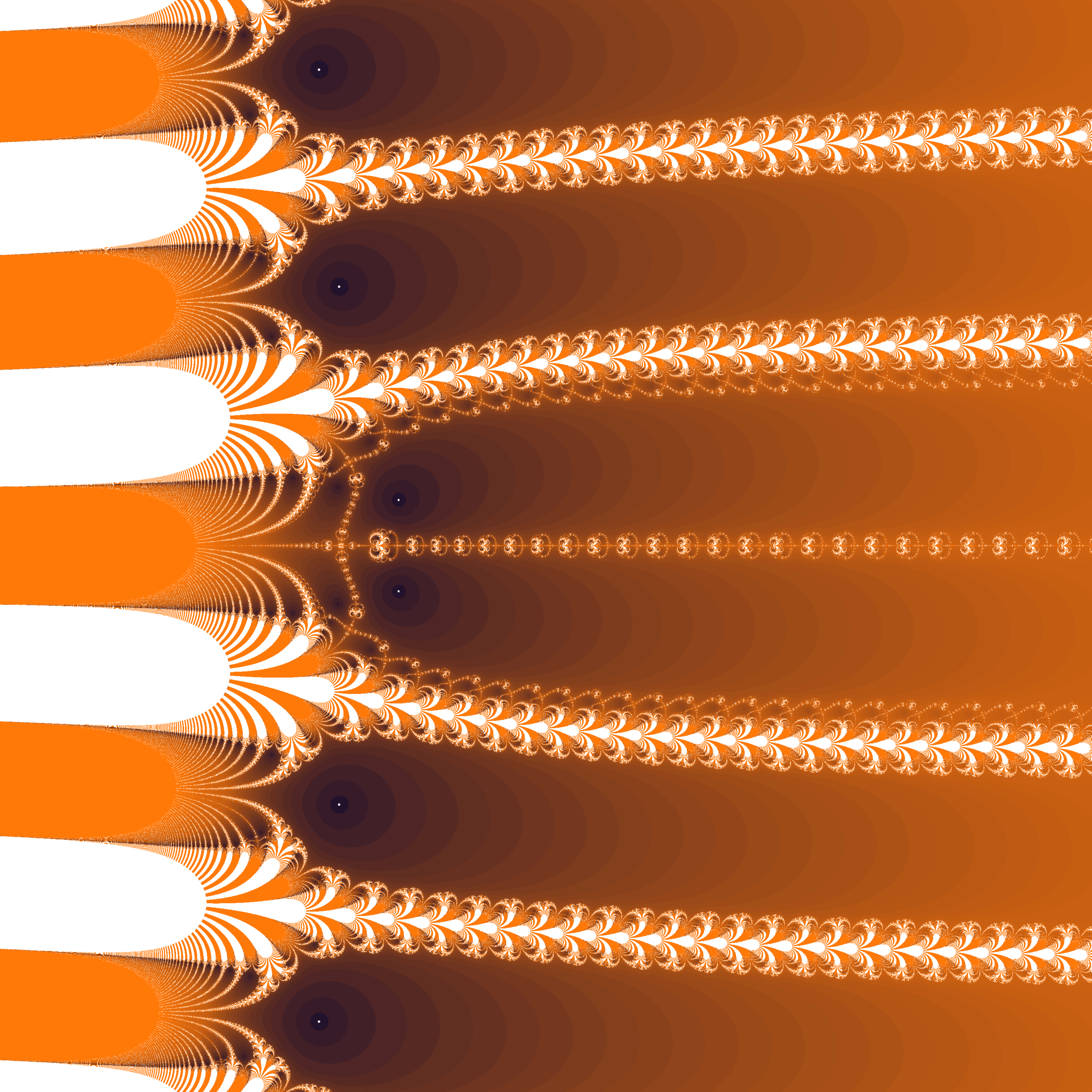}
\hfill
\includegraphics[width=0.48\textwidth]{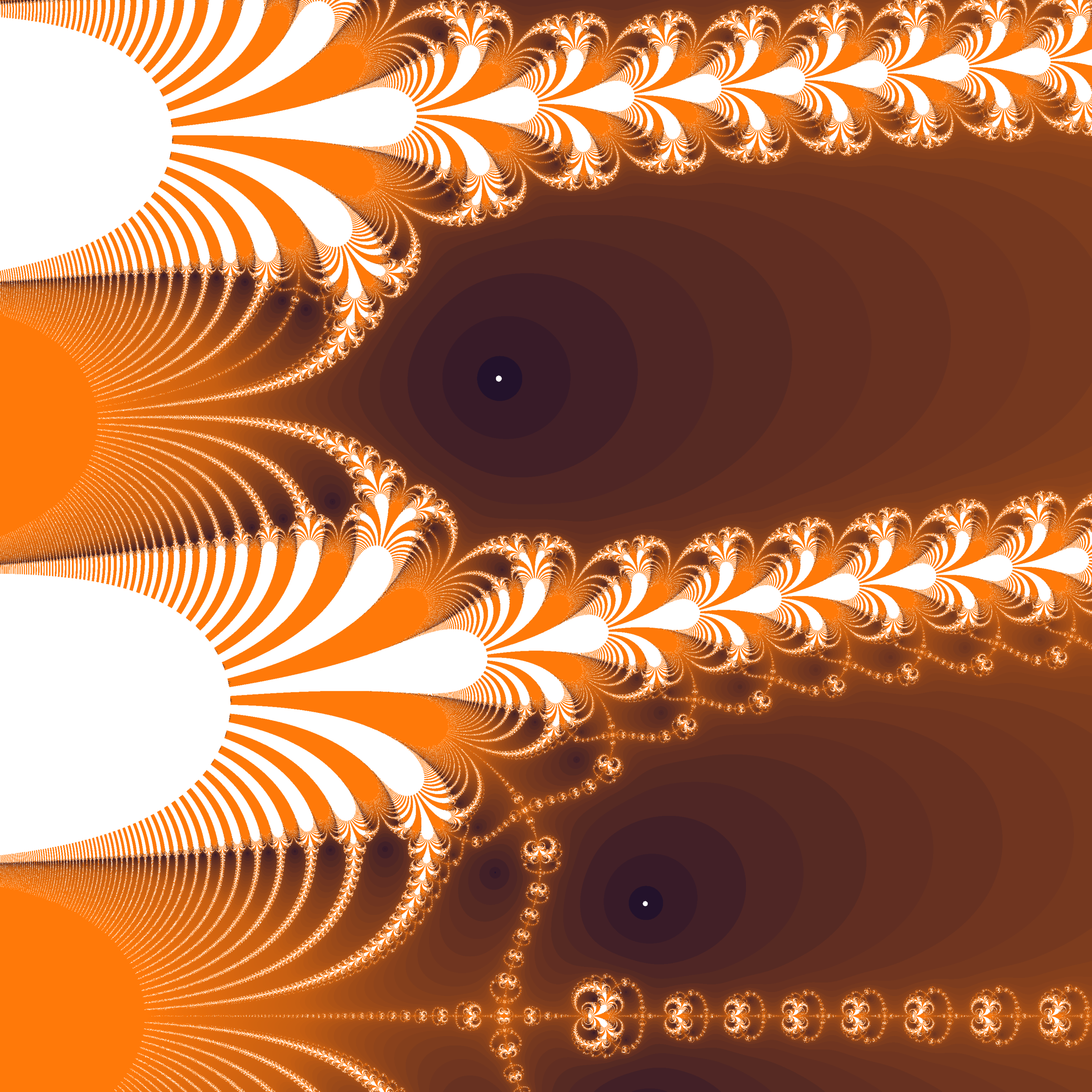}
\caption{\small \label{brown} Left: The dynamical plane of the map $N(z) = z - (z+e^{-z})/(z+1)$, Newton's method of $F(z)=1+z e^z$, showing the accesses to infinity from attracting basins of the zeroes of $F$. Right: Zoom of the dynamical plane.}
\end{figure}

Note that for transcendental entire maps, invariant Fatou components are always simply connected (see \cite{baker75}). In this setup, the investigation of the relations between the dynamics of $f$, the boundary behaviour of a Riemann map and the geometry of the boundary of an invariant Fatou component $U$ (in particular the question of the accessibility of the point at infinity and a possible number of different accesses to this point) has been a subject of interest for a long time, in which one finds a beautiful interplay between dynamics, analysis and topology. 
An early surprising example was given by Devaney and Goldberg \cite{devaney-goldberg} where they considered hyperbolic exponential maps $E(z)= \lambda e^z$ with a completely invariant basin of attraction $U$, whose boundary has infinitely many accesses to infinity. In fact, they studied the self-map of the disc $g=\varphi^{-1} \circ E \circ \varphi$ (called an {\em inner function associated to} $E|_U$), where $\varphi:\D\to U$ is a Riemann map. They showed that $\varphi$ has a radial limit at all points of $\bdd$, and those which correspond to accesses to infinity from $U$ form a dense set. Their results were later generalized by Bara\'nski and
Karpi\'nska \cite{coding} and Bara\'nski \cite{hairs} in the case of completely invariant basins of attraction of entire maps from the Eremenko--Lyubich class $\mathcal{B}$ of disjoint type (possibly of finite order). In the general setup of transcendental entire maps, the subject was further explored by Baker and Weinreich \cite{bw}, Baker and Dom\'{\i}nguez \cite{baker-dominguez} and Bargmann \cite{bargmann}. In particular, in \cite{baker-dominguez} it was proved that if $U$ is an invariant Fatou component of an transcendental entire map, which is not a univalent Baker domain, then infinity has either none or infinitely many accesses from $U$. 

In this paper we consider the accessibility of the point at infinity in the setup of transcendental meromorphic functions, i.e.~maps with preimages of the essential singularity at infinity (poles). Up to now, this subject has been almost untouched in the literature. The setting is different in many essential aspects from both the entire and rational case.  On one hand, infinite degree of the map on an invariant Fatou component makes the dynamics more complicated --  both for the map $f$ and for the inner function $g$ associated to $f$ on the given invariant component. On the other hand, a possible presence of poles seems to ``simplify'' the topology of the components' boundaries. Our goal in this paper is to explore these various relationships.  

\section{Main Results}

We start by giving a precise definition of the notion of an access to a boundary point from a planar domain. Let $U$ be a simply connected domain in the complex plane $\mathbb C$. By $\bd U$ we denote the boundary of $U$ in the Riemann sphere $\clC$. 

\begin{defn}[{\bf Access to boundary point}{}] 
A point $v\in \partial U$ is {\em accessible} from $U$, if there exists a curve  
$\gamma:[0,1] \to \clC$ such that $\gamma([0,1)) \subset U$ and $\gamma(1) = 
v$. We also say that $\gamma$ {\em lands} at $v$. (Equivalently, one can consider curves $\gamma:[0,1) \to U$ such that $\lim_{t \to 1^-} \gamma(t) = v$).

Fix a point $z_0 \in U$ and an accessible point $v\in \partial U$. A homotopy 
class (with fixed endpoints) of curves $\gamma:[0,1] \to \clC$ such that 
$\gamma([0,1)) \subset U$, $\gamma(0) = z_0$, $\gamma(1) = v$ is called an {\em 
access} from $U$ to the point $v$ (or an access to $v$ from $U$).
\end{defn}

The choice of the point $z_0$ is irrelevant in the sense that for any other point $z_0' \in U$ there is a one-to-one correspondence between accesses defined by curves starting from $z_0$ and from $z_0'$, given by $\gamma \longleftrightarrow \gamma_0 \cup \gamma$, where $\gamma_0$ is a curve connecting $z_0$ to $z_0'$ in $U$ (with a suitable parameterization). For convenience, throughout the paper we assume $z_0 = \varphi(0)$,
where 
$$\varphi : \D \to U$$ 
is a Riemann map from the open unit disc onto $U$.  

Every point $\zeta\in\partial \D$ for which the radial limit $\lim_{t\to 1^-} \varphi(t \zeta)$ of $\varphi$ exists and equals a point $v\in\partial U$ defines an access to $v$ from $U$ in the obvious way. But a stronger correspondence is true. 
\begin{cortheorem}
Let $U\subset \C$ be a simply connected domain and let $v \in \bd U$. Then there is a one-to-one correspondence between accesses from $U$ to $v$ and points $\zeta \in \bdd$, such that the radial limit of $\varphi$ at $\zeta$ exists and is equal to $v$. The correspondence is given as follows.
\begin{enumerate}[\rm (a)]
\item If $\A$ is an access to $v \in\bd U$, then there is a point $\zeta \in \bdd$, such that the radial limit of $\varphi$ at $\zeta$ is equal to $v$,  and for every $\gamma \in \A$, the curve $\varphi^{-1}(\gamma)$ lands at $\zeta$. Moreover, different accesses correspond to different points in $\partial \D$.
\item If the radial limit of $\varphi$ at a point $\zeta 
\in \bdd$ is equal to $v \in \bd U$, then there exists an access $\A$ to $v$, such that for every curve $\eta \subset \D$ landing at $\zeta$, if $\varphi(\eta)$ lands at some point $w \in\clC$, then $w = v$ and $\varphi(\eta) \in \A$.
\end{enumerate}
Equivalently, there is a one-to-one correspondence between accesses to $v$ from $U$ and the prime ends of $\varphi$ of the first or second type, whose impressions contain $v$ as 
a principal point. 
\end{cortheorem}
Although this is a folklore result used in several papers, we were unable to find a written reference for it. Hence, for completeness we include its proof in Section~\ref{sec:accesses}.

Now suppose that 
\[
f:U\to U
\]
is a holomorphic map. Then the map $f$ induces some dynamics on the set of  accesses to boundary points of $U$.

\begin{defn}[{\bf Invariant access}{}]  An access $\A$ to $v$ is \emph{invariant}, if there exists $\gamma \in \A$, such that $f(\gamma) \cup \eta \in \A$, where $\eta$ is a curve  connecting $z_0$ to $f(z_0)$ in $U$. An access $\A$ to $v$ is \emph{strongly invariant}, if for every $\gamma \in \A$, we have $f(\gamma) \cup \eta \in \A$.
\end{defn}

\begin{rem*} Since $U$ is simply connected, the choice of the curve $\eta$ is irrelevant.
\end{rem*}

\begin{rem*} As we consider two kinds of invariance of accesses, our terminology differs slightly from the one in \cite{johannes}. An invariant access in the sense of \cite{johannes} is called here strongly invariant.
\end{rem*}

As an example, observe that every simply connected invariant Baker domain $U$ of an entire or meromorphic map $f$ has an invariant access to infinity. Indeed, given any point $z\in U$ and a curve $\gamma$ joining $z$ and $f(z)$ within $U$, the curve $\Gamma:=\bigcup_{n\geq 0} f^n(\gamma)$ is unbounded and lands at infinity, defining an access to this point from $U$, which we call the {\em dynamical access} to infinity from $U$. Since $f(\Gamma) = \bigcup_{n\geq 1} f^n(\gamma) \subset \Gamma$ also lands at infinity, this access is invariant. Similarly, every simply connected invariant parabolic basin $U$ of an entire or meromorphic map $f$ has an invariant dynamical access to a parabolic fixed point $v \in \bd U$. 

\begin{rem*} 
In the same way one can define periodic and strongly periodic accesses to  boundary points of $U$. In the case of a rational map $f$, the analysis of periodic accesses to boundary points is strongly related to the study of landing periodic rays in simply connected invariant Fatou components, e.g.~basins of infinity for polynomials, see e.g.~\cite{carlesongamelin,milnor} and the references therein.
\end{rem*}

The notions of invariance and strong invariance of accesses do not coincide, as shown by the following example.

\begin{ex*}[{\bf Invariant access which is not strongly invariant}{}]
Let $f:\C \to \clC$, $f(z) = z + \tan z$ and $U  = \{z \in \C \mid \Im(z) > 0\}$. Then $U$ is an invariant Baker domain of $f$ and the dynamical access to $\infty$ defined by the homotopy class of the curve $i\R_+$ is invariant, but not strongly invariant (see Example~\ref{ex1} for details). Another example of an invariant access which is not strongly invariant is given in Example~\ref{ex2}.
\end{ex*}

\begin{ex*}[{\bf Strongly invariant access}{}]
Let $f:\C \to \clC$, $f(z)=z-\tan z $, the Newton's method of $F(z)=\sin z$. Then the Fatou set of $f$ contains infinitely many basins of attraction $U_k$, $k\in\Z$ of superattracting fixed points $k\pi$, and every $U_k$ has two strongly invariant accesses to infinity. See Example~\ref{ex5} for details. 
\end{ex*}

The following proposition gives a precise condition for an invariant access to be  strongly invariant. The proof is included in Section~\ref{sec:TheoremE}.

\begin{prop}[{\bf Characterization of strongly invariant accesses}{}] \label{strongly}
Let $U\subset \C$ be a simply connected domain and let $f:U\to U$ be a holomorphic map. Suppose $\cala$ is an invariant access from $U$ to a point $v\in \partial U$. Then $\cala$ is strongly invariant if and only if for every $\gamma \in\cala$, the curve $f(\gamma)$ lands at some point in $\partial U$.
\end{prop}

From now on, assume that 
\[
f: \C \to \clC
\]
is meromorphic and $U \subset \C$ is a simply connected invariant Fatou component of $f$. Although our results hold also for rational maps, we are mainly interested in the case when $f$ is transcendental. As above, $\varphi:\D\to U$ is a Riemann map. Let 
\[
g: \D \to \D, \qquad  g = \varphi^{-1} \circ f \circ \varphi,
\]
be the \emph{inner function associated to} $f|_U$. Then $g$ has radial limits belonging to $\bdd$ at almost every point of $\bdd$ (see Section~\ref{sec:prelim}). 

Our first result shows a relation between the dynamics of the inner function $g$ and the dynamics of $f$ on the set of accesses of $U$. More precisely, we shall relate invariant accesses in $U$ and fixed points of $g$ in $\partial \D$. Since $g$ does not necessarily extend to $\partial \D$, we use a weaker concept of a (\emph{radial}) {\em boundary fixed point} of $g$, i.e.~a point $\zeta \in \bdd$, such that the radial limit of $g$ at $\zeta$ is equal to $\zeta$. The Julia--Wolff Lemma (Theorem~\ref{jw}) ensures that at such points the \emph{angular derivative} $\lim_{t \to 1^-} (g(t\zeta) - \zeta)/((t-1)\zeta)$ of $g$ exists and belongs to $(0,+\infty) \cup \{\infty\}$ (see Section~\ref{sec:prelim}). The finiteness of the angular derivative turns out to be essential. 

\begin{defn}[{\bf Regular boundary  fixed point}{}] We say that a boundary fixed point $\zeta \in \bdd$ of a holomorphic map $g: \D \to \D$ is \emph{regular}, if the angular derivative of $g$ at $\zeta$ is finite.
\end{defn}

Our first result is the following dynamical version of the Correspondence Theorem.
Recall that a fixed point $z$ of a holomorphic map $f$ is called weakly repelling, if $|f'(z)| > 1$ or $f'(z) = 1$.

\begin{thmA}
Let $f:\C\to\chat$ be a meromorphic map and $U\subset \C$ a simply connected invariant Fatou component. Let $\varphi:\D\to U$ be a Riemann map and $g= \varphi^{-1} \circ f \circ \varphi :\D\to\D$ the inner function associated to $f|_U$. Then the following hold.
\begin{enumerate}[\rm (a)]
\item If $\cala$ is an invariant access from $U$ to a point $v \in \bd U$, then $v$ is either infinity or a fixed point of $f$ and $\A$ corresponds to a boundary fixed point $\zeta  \in \partial \D$ of $g$. 
\item If $\zeta \in \partial \D$ is a regular boundary fixed point of $g$, then the radial limit of $\varphi$ at $\zeta$ exists and is equal to $v$, where $v$ is either infinity or a weakly repelling fixed point of $f$ in $\partial U$. Moreover, $\zeta$ corresponds to an invariant access $\A$ from $U$ to $v$.
\end{enumerate}
\end{thmA}

\begin{rem*}
If $\zeta\in\bdd$ is an irregular  boundary fixed point of $g$ (i.e.~$g$ has infinite angular derivative at $\zeta$) and $\varphi$ has a radial limit at $\zeta$ equal to a point $v\in\partial U$, then $\zeta$ still corresponds to an access from $U$ to $v$, but this access may be not invariant. To ensure invariance, one would need the unrestricted limit of $\varphi$ to exist or, equivalently, the impression of the prime end $\zeta$ to be the singleton $\{v\}$. In particular, this holds when $\bd U$ is locally connected.
\end{rem*}

\begin{rem*}
If $g$ extends holomorphically to a neighbourhood of the boundary fixed point $\zeta$, then one can show that $\cala$ contains an invariant curve (in fact many of them). More precisely, there is a curve $\gamma$ in $U$ landing at $\infty$ or at a weakly repelling fixed point of $f$ in $\bd U$, such that $f(\gamma) \supset \gamma$ (in the case when $\zeta$ is repelling) or $f(\gamma) \subset \gamma$ (in the case when $\zeta$ is attracting or parabolic). Such a curve can be constructed by the use of suitable local coordinates near $\zeta$ (see e.g.~\cite{carlesongamelin, milnor}). 
\end{rem*}

In order to use Theorem~A for describing the set of accesses to infinity from $U$, we need the following notion of a singularly nice map. 

\begin{defn}\label{defn:sparse}  A point $\zeta^* \in \bd \D$ is a {\em singularity }of $g$, if $g$ cannot be extended holomorphically to any neighbourhood of $\zeta^*$ in $\C$.  
We say that $f|_U$ is {\em singularly nice} if there exists a singularity $\zeta^*\in\bdd $ of the inner function $g$ associated to $f|_U$, such that the angular derivative of $g$ is finite at every point $\zeta$ in some punctured neighbourhood of $\zeta^*$ in $\bdd$. 
\end{defn}

\begin{rem*}
If $g$ has an isolated singularity in $\bdd$, then $f|_U$ is singularly nice. Indeed, in this case $g$ extends holomorphically to a neighbourhood of any point $\zeta \in \bdd$, $\zeta \neq \zeta^*$ close to $\zeta^*$, so the derivative of $g$ exists and is finite at $\zeta$.
\end{rem*}

Note that by definition, if $f|_U$ is singularly nice, then $\deg f|_U$ (and hence $\deg g$) is infinite. If $\deg f|_U$ is finite, then $g$ is a finite Blaschke product and extends by the Schwarz reflection to the whole Riemann sphere. In this case, $g$ has finite derivative at every point $\zeta \in \bdd$ (see Section~\ref{sec:prelim}).

The following proposition is proved in Section~\ref{sec:TheoremE}. and gives a useful condition to ensure that $f|_U$ is singularly nice.

\begin{prop}[{\bf Singularly nice maps}{}]\label{prop:sing}
Let $f:\C \to \clC$ be a meromorphic map and $U \subset \C$ a simply connected invariant Fatou component such that $\deg f|_U = \infty$. If there exists a non-empty open set $W \subset U$, such that for every $z \in W$ the set $f^{-1}(z) \cap U$ is contained in the union of a finite number of curves in $U$ landing at some points of $\bd U$, then $f|_U$ is singularly nice. 
\end{prop}

 In Section~\ref{sec:examples} we shall see several applications of this criterion.

Define the sets 
\begin{align*}
\IA(U)&= \{\text{invariant accesses from $U$ to its boundary points}\}\\
\IA(\infty, U)&= \{\text{invariant accesses from $U$ to infinity}\}\\
\IA({\rm wrfp}, U)&=\{\text{invariant accesses from $U$ to weakly repelling fixed points of $f$ in $\partial U$}\}.
\end{align*}

As the first application of Theorem~A, we prove the following result.

\begin{thmB}
Let $f:\C \to \clC$ be a meromorphic map and let $U \subset \C$ be a  simply connected invariant Fatou component of $f$. Let $\varphi:\D\to U$ be a Riemann map and $g= \varphi^{-1} \circ f \circ \varphi:\D\to\D$ the inner function associated to $f|_U$. Set $d = \deg f|_U$. Then the following statements hold.
\begin{enumerate}[\rm (a)]
\item If $ d< \infty$, then 
\[
\IA(U) = \IA(\infty, U) \cup \IA({\rm wrfp}, U)
\]
and $\IA(U)$ has exactly $D$ elements, where $D$ is the number of fixed points of $g$ in $\bdd$. Moreover, $D \ge 1$ $($unless $U$ is an invariant Siegel disc$)$ and $d - 1 \le D \le d + 1$.  
\item
If $d = \infty$ and $f|_U$ is singularly nice, then the set $\IA(\infty, U) \cup \IA({\rm wrfp}, U)$ is infinite. In particular, if $U$ has only finitely many invariant accesses to infinity, then $f$ has a weakly repelling fixed point in $\bd U$ accessible from $U$.

\item If $U$ is bounded, then $f$ has a weakly repelling fixed point in $\bd U$ accessible from~$U$ or $U$ is an invariant Siegel disc.
\end{enumerate}
\end{thmB}

\begin{rem*}
By Theorem~B, if $d = \infty$, $U$ has only finitely many invariant accesses to infinity and $f$ has only finitely many weakly repelling fixed points in $\partial U$, then at least one of the points must have infinitely many invariant accesses from $U$. However, we do not know of any example where this is the case. On the other hand, the map $f(z)=z+\tan(z)$ provides an example of a singularly nice Baker domain with only one access to infinity and infinitely many weakly repelling (accessible) fixed points of $f$ in $\partial U$ (see Example~\ref{ex1} for details). 
\end{rem*}

Now we apply the above results to the Newton maps. Recall that a Newton map $N$ is Newton's method for finding zeroes of an entire function $F$, 
\[                                                                                                                                                                                    N(z) = z - \frac{F(z)}{F'(z)}.                                                                                                                                                                                      \]
This class is of particular interest in this context, since for  Newton maps, all Fatou components are simply connected, as proved in \cite{bfjknew,bfjk,shishikura}. Moreover, all fixed points of a Newton map are attracting (except for the repelling fixed point $\infty$ for rational Newton maps). Hence, Newton maps have no invariant parabolic basins or Siegel discs and no finite fixed points in their Julia sets. In particular there are no finite weakly repelling fixed points. This fact together with Theorem~B immediately implies the following corollary.

\begin{corC}
Let $N:\C \to \clC$ be the Newton's method for an entire function and let $U \subset \C$ be an invariant Fatou component of $N$. Let $\varphi:\D\to U$ be a Riemann map and $g= \varphi^{-1} \circ N \circ \varphi:\D\to\D$ the inner function associated to $N|_U$. Set $d= \deg N|_U$. Then the following statements hold.
\begin{enumerate}[\rm (a)]
\item If $ d< \infty$, then there are no invariant accesses from $U$ to points $v \in \bd U \cap \C$ and exactly $D$ invariant accesses from $U$ to infinity, where $D$ is the number of fixed points of $g$ in $\bdd$. Moreover, $D \ge 1$ and $d - 1 \le D \le d + 1$.  
\item If $d= \infty$ and $N|_U$ is singularly nice, then there are infinitely many invariant accesses to infinity from $U$. 
\item $U$ is unbounded.

\end{enumerate}
\end{corC}

See Section~\ref{sec:examples} for several illustrating examples.

\begin{rem*} The statement~(c) was previously proved in \cite{mayer}. 
Note that Corollary~C shows that infinity is accessible from an invariant Fatou component $U$ of a Newton map $N$, unless $\deg N|_U$ is infinite and $N|_U$ is not singularly nice. It remains an open question, whether infinity is always accessible from $U$. 
\end{rem*}

In the case when a Newton map has a completely invariant Fatou component, much more can be said about accesses in the remaining invariant components. 

\begin{thmD}
Let $N:\C \to \clC$ be a Newton map with a completely invariant Fatou component $V \subset \C$ and let $U \subset \C$, $U \neq V$ be an invariant Fatou component of $f$. Set $d=\deg N|_U$. Then 
\[
d  \in \{1, 2, \infty\}                                                                                                                                                                                                                                                                                 \]
and for every $v \in \bd U$ there is at most one access to $v$ from $U$. Moreover,
\begin{enumerate}[\rm (a)]
\item if $d \in \{1, 2\}$, then $U$ has a unique access $\A$ to $\infty$ and $\A$ is invariant,
\item if $d = 1$, then $\bd U$ does not contain a pole of $N$ accessible from $U$,
\item if $d = 2$, then $\bd U$ contains exactly one accessible pole of $N$.
\end{enumerate}
\end{thmD}

Note that an invariant Fatou component of a transcendental Newton map can have accesses to infinity which are not invariant, as shown by the example in Figure~\ref{fig:bd}. This is an important difference with respect to rational Newton maps, where all accesses to infinity from the attracting basins are (strongly) invariant  \cite{hubbardschleicher}. 
This phenomenon is related to the existence of accessible poles of $f$ in the boundary of the component, as shown by the following theorem for general meromorphic maps.

\begin{thmE} 
Let $f:\C\to\chat$ be a meromorphic map and $U \subset \C$  a simply connected invariant Fatou component, such that infinity is accessible from $U$. Set $d=\deg f|_U$.
\begin{enumerate}[\rm (a)]
\item If $1 < d < \infty$ and $\partial U \cap \C$ contains no poles of $f$ accessible from $U$, then $U$ has infinitely many accesses to infinity, from which at most $d+1$ are invariant.
\item If $d = \infty$ and $\partial U$ contains only finitely many poles of $f$ accessible from $U$, then $U$ has infinitely many accesses to infinity. 
\end{enumerate}
\end{thmE}

The proof of Theorem~E uses the following proposition, which is analogous to \cite[Theorem 1.1]{baker-dominguez}. The proof is contained in Section~\ref{sec:TheoremE}.

\begin{prop}\label{poles}
Let $f:\C\to\chat$ be a meromorphic map and $U \subset \C$ a simply connected invariant Fatou component. Let $\varphi:\D\to U$ be a Riemann map and $g= \varphi^{-1} \circ f \circ \varphi:\D\to\D$ the  inner function associated to $f|_U$. Suppose $\infty$ is accessible from $U$ and let 
\[
\Theta = \{\zeta \in\partial \D \mid  \text{the radial limit of $\varphi$ at $\zeta$ is  $\infty$ or a pole of $f$}\}.
\]
Then 
\[
\Sing(g) \subset \Acc(\Theta),
\]
where $\Sing(g)$ denotes the set of singularities of $g$ and $\Acc(\Theta)$ is the set of accumulation points of $\Theta$.
\end{prop}

The paper is organized as follows. Section~\ref{sec:prelim} contains preliminaries. In Section~\ref{sec:accesses} we include the proof of the Correspondence Theorem, while Theorems~A and~B are proved in Section~\ref{sec:proofs}. Theorems~D and~E and the remaining facts are proved in Section~\ref{sec:TheoremE}. Finally, in Section~\ref{sec:examples}, we apply our results to study a number of examples of meromorphic maps showing a diversity of phenomena related to accesses to boundary points of invariant Fatou components. 

\subsection*{Acknowledgements}

We wish to thank Institut de Matem\`atica de la Universitat de Barce\-lona (IMUB) and the Institute of Mathematics of the Polish Academy of Sciences (IMPAN) for their hospitality while this work was in progress. We are grateful to Lasse Rempe, Phil Rippon and Gwyneth Stallard for interesting questions and comments. 

%%%%%%%%%%%%%%%%%%%%%%%%%
%%%%%%%%%%%%%%%%%%%%%%%%%
\section{Preliminaries} \label{sec:prelim}

In this section we state some classical results from complex analysis and topology, which we use in our proofs.

\subsection{Holomorphic maps on the unit disc $\D$ and their boundary behaviour}  

\

\smallskip

For a general exposition on this wide field of research refer e.g.~to \cite{col-loh,garnett,pommerenke-book,shapiro} and the references therein.

Let $U$ be a simply connected domain in $\mathbb C$ and let $\varphi : \D \to U$ be a Riemann map onto $U$. Caratheodory's Theorem (see e.g.~\cite[Theorems~2.1 and~2.6]{pommerenke-book}) states that $\bd U$ is locally connected if and only if $\varphi$ extends continuously to $\overline{\D}$. But even when this is not the case, radial limits of $\varphi$ exist at almost all points of $\bdd$. This is known as Fatou's Theorem (see e.g.~\cite[Theorem 1.3]{pommerenke-book}).

\begin{thm}[{\bf Fatou Theorem}{}] For almost every $\zeta \in \bdd$ there 
exists the \emph{radial limit} 
$$
\lim_{t\to 1^-} \varphi(t\zeta)
$$
of $\varphi$ at $\zeta$. 
Moreover, if we fix $\zeta\in \bdd$ so that this limit exists, then for almost 
every $\zeta'\in \bdd$ the radial limit of $\varphi$ at $\zeta'$ is different from the radial limit at $\zeta$.
\end{thm}

\begin{rem*} Although Fatou's Theorem is stated for univalent maps, it is also true for general bounded analytic maps of $\D$, see e.g.~\cite[Section 6.1 and p.~139]{pommerenke-book}.
\end{rem*}
 
%Recall that for a holomorphic map on $\D$ it is equivalent  to have the radial limit equal to $v$ at $\zeta\in\bdd$,
%and to have the \emph{angular limit} equal to $v$ at $\zeta$ (the limit for $z$ tending to $\zeta$ in a Stolz angle at $\zeta$). 
 
Another classical result about Riemann maps is the following (see e.g.~\cite[Theorem 2.2]{carlesongamelin}).

\begin{thm}[{\bf Lindel\"of Theorem}{}]
Let $\gamma:[0,1) \to U$ be a curve which lands at a point $v \in \partial 
U$. Then the curve $\varphi\inv(\gamma)$ in $\D$ lands at some point $\zeta \in 
\partial \D$. Moreover, $\varphi$ has the radial limit at $\zeta$ equal to $v$. In 
particular, curves that land at different points in $\partial U$ correspond to curves which land at different points of $\partial \D$.
\end{thm}

The assertion of the Lindel\"of Theorem holds also for non-univalent holomorphic maps on the unit disc, as shown by the Lehto--Virtanen Theorem (see e.g.~\cite{LVActa} or \cite[Section 4.1]{pommerenke-book}). 
\begin{thm}[{\bf Lehto--Virtanen Theorem}{}] 
Let $h: \D \to \clC$ be a normal holomorphic map $($e.g.~let $h$ be a holomorphic  map on $\D$ omitting three values in $\clC)$ and let $\gamma$ be a curve in $\D$ landing at a point $\zeta \in\bdd$. Then, if $h(z) \to v$ as $z \to \zeta$ along $\gamma$, then $h$ has the angular limit at $\zeta$ (i.e. the limit for $z$ tending to $\zeta$ in a Stolz angle at $\zeta$) equal to $v$.  In particular, radial and angular  limits are the same. 
\end{thm}

As noted in the introduction, the concept of angular derivative plays an important role in our discussion.

\begin{defn} Let $h$ be a holomorphic map of $\D$. We say that $h$ has an {\em angular derivative} $a \in \clC$ at a point $\zeta \in \bdd$, if $h$ has a finite radial limit $v$ at $\zeta$ and 
\[
\frac{h(z) - v}{z - \zeta}
\]
has the radial limit at $\zeta$ equal to $a$.
\end{defn}

Now we turn to the special situation where $h:\D\to \D$ is a holomorphic self-map of the unit disc. 
It is a remarkable fact that for such maps, the angular derivative at a boundary point exists (finite or infinite) as soon as the radial limit at that point exists. This is known as the Julia--Wolff Lemma (see e.g.~\cite[Proposition 4.13]{pommerenke-book}).

\begin{thm}[{\bf Julia--Wolff Lemma}{}] \label{jw}
If $h: \D \to \D$ is a holomorphic map which has a radial limit $v$ at a point $\zeta \in \bdd$, then $h$ has an angular derivative at $\zeta$ equal to $a$ such that
\[
0 < \zeta\frac{a}{v} \le +\infty.
\]
In particular, if the radial limit of $h$ at $\zeta$ is equal to $\zeta$, then the angular derivative of $h$ at $\zeta$ is either infinite or a positive real number.
\end{thm}

The Denjoy--Wolff Theorem (see e.g.~\cite[Theorem 3.1]{carlesongamelin}) describes the different possibilities for the dynamics of such maps.

\begin{thm}[{\bf Denjoy--Wolff Theorem}{}] 
\label{thm:denjoy}
Let $h: \D \to \D$ be a non-constant
non-M\"obius holomorphic self-map of $\D$. Then there exists a point $\zeta \in \overline{\D}$, called
the \emph{Denjoy--Wolff point} of $h$, such that the iterations $h^n$ tend to $\zeta$  as $n \to \infty$
uniformly on compact subsets of $\D$.
\end{thm}

The Denjoy--Wolff point is a special case of a boundary fixed point.
\begin{defn} A point $\zeta \in \bdd$ is a \emph{boundary fixed point} of a holomorphic map $h: \D \to \D$ if the radial limit of $h$ at $\zeta$ exists and is equal to $\zeta$. Such a point is also called \emph{radial} or \emph{angular boundary fixed point}.
\end{defn}

By the Julia--Wolff Lemma, the angular derivative of $h$ exists at every boundary fixed point of $h$. The finiteness of this derivative has a significant influence on the boundary behaviour of $h$ near $\zeta$. 

\begin{defn} \label{radialrepelling}
Let $\zeta \in \bdd$ be a boundary fixed point of $h: \D \to \D$. Then $\zeta$ is called \emph{regular}, if the angular derivative of $h$ at $\zeta$ is finite.
\end{defn}

Boundary fixed points can be classified as follows.

\begin{defn}
We say that $\zeta$ is, respectively, \emph{attracting}, \emph{parabolic}, or \emph{repelling}, if the angular derivative of $h$ at $\zeta$ is, respectively, smaller than $1$, equal to $1$, or greater than $1$ or infinite. 
\end{defn}

The following is the second part of the Denjoy--Wolff Theorem (see e.g.~\cite[Theorem 2.31]{bargmann}).

\begin{thm}\label{DWII}
Let $h:\D\to\D$ be a non-constant
non-M\"obius holomorphic self-map of $\D$ and let $\zeta\in \overline{\D}$ be the 
Denjoy--Wolff point of $h$. Then $\zeta$ is the unique boundary fixed point of $h$ which is attracting or parabolic.
\end{thm}

A special class of holomorphic self-maps of $\D$ is given by the class of inner functions.
\begin{defn} An \emph{inner function} is a holomorphic function $h:\D\to \D$ such that the radial limit of $h$ at $\zeta$ has modulus $1$ for almost every $\zeta \in \bdd$.
\end{defn}

If $U$ is a simply connected invariant Fatou component of a meromorphic function $f$ and $\varphi:\D \to U$ is a Riemann map, it is easy to check that the map $g:=\varphi^{-1}\circ f \circ \varphi$ is always an inner function, which justifies the terminology used in the introduction. 

\begin{defn} Let $h:\D\to\D$ be an inner function. A point $\zeta \in \partial \D$ is a {\em singularity} of $h$ if 
the map $h$ cannot be continued analytically to any neighbourhood of $\zeta$ in $\C$. The set of all singularities of $h$ is denoted by $\Sing(h)$.
\end{defn}

The following proposition describes the topological degree of a meromorphic map on a simply connected domain (see \cite[Prop. 2.8]{bfr} or \cite{hei57,heins86}). Although in this references it is stated for entire functions, the proof goes through for meromorphic maps as well. 

\begin{prop} \label{degree}
Let $f$ be a meromorphic map, let $U\subset\C$ be a simply connected domain and
   let $\widetilde{U}$ be a component of $f^{-1}(U)$. Then either
\begin{enumerate}[{\normalfont(a)}]
  \item \label{item:inner_proper}
    $f\colon\widetilde{U}\to U$ is
   a proper map $($and hence has finite degree$)$, or 
  \item \label{item:inner_infinite}
    $f^{-1}(z)\cap \widetilde{U}$ is infinite
     for every $z\in U$ with at most one exception. 
\end{enumerate}
In case~\ref{item:inner_infinite}, either $\widetilde{U}$ contains an asymptotic curve corresponding to an asymptotic value in $U$, or $\widetilde{U}$ contains infinitely many critical points.
\end{prop}

If an inner function $h: \D \to \D$ is a proper map (i.e.~has a finite degree) then $h$ is a finite Blaschke product of the form
\[
h(z)=e^{i\theta}\prod_{k=1}^d  \frac{z-a_k}{1-\overline{a_k}z}
\]
for some $a_1, \ldots, a_d \in \D$, $\theta\in [0,2\pi)$, and therefore $h$ has no singularities and extends to the whole Riemann sphere. Note that since $\D$ is invariant under $h$, the map cannot have critical points on $\partial \D$. 
Moreover, $h$ has $d+1$ 
fixed points on $\clC$, counted with multiplicity and at most one of them is in $\D$. Among the fixed points on $\bdd$, at most one is not repelling, by Theorem~\ref{DWII}.

If $h$ is not proper, then it has infinite degree and has at least one singularity on $\partial \D$. In this case there are infinitely many repelling fixed 
points of $h$ in $\bdd$, as shown by the following theorem, together with other properties of singularities which will be used in our proofs.

\begin{thm}[{\bf Singularities of an inner function}{}] \label{thm:bargmann}
Let $h: \D \to \D$ be an inner function and let $\zeta^* \in \bdd$ be a singularity of $h$. Then the following hold. 
\begin{enumerate}[\rm (a)]
\item  $\zeta^*$ is an accumulation point of the set of  boundary repelling fixed points of $h$. 
\item Given any $\zeta \in \bdd$, the singularity $\zeta^*$ is an accumulation point of the set of points $\zeta' \in \bdd$, for which there is a path $\eta$ in  $\D$ landing at $\zeta'$, such that $h(\eta)$ lands at $\zeta$ in a Stolz angle.
\item $\Sing(h)$ is the set of accumulation points of the set $h^{-1}(z)$ for almost every point $z \in \D$.
\end{enumerate}
\end{thm}

The statement~(a) is proved in \cite{bargmann} as Theorem~2.32, while the statement~(b) is \cite[Lemma~5]{baker-dominguez}. The statement~(c) follows from standard facts on inner functions, see e.g.~\cite[Theorems~II.6.1 and~II.6.4]{garnett}. 
 
We end this section by recalling the Wolff Boundary Schwarz Lemma (see e.g. \cite[p. 81]{shapiro} or \cite[Lemma 2.33]{bargmann}).

\begin{thm}[\bf Wolff Boundary Schwarz Lemma{}]
If $h: \D \to \D$ is a holomorphic map with the Denjoy--Wolff point $\zeta \in \bdd$ and $V \subset \D$ is an open disc tangent to $\bdd$ at $\zeta$, then $h(V) \subset V$. If, additionally, $h$ is a non-M\"obius inner function, then  $h(\overline V \setminus \{\zeta\}) \subset V$.
\end{thm}

\subsection{Topology} 

\ 

\smallskip

\begin{defn} By a \emph{topological arc} we mean a set homeomorphic to the closed interval $[0,1]$, and by a \emph{Jordan curve}  a set homeomorphic to a circle.
\end{defn}

We will frequently use (without noting) the Jordan Theorem stating that a planar Jordan curve separates the plane into two connected components.

\begin{defn}A path-connected topological space $X$ is \emph{simply connected}, if every loop in $X$ is homotopic to a point in $X$. An equivalent condition is that every two curves in $X$ connecting points $x, y \in X$ are homotopic with endpoints fixed.
\end{defn}

The following facts are standard results in planar topology (see e.g.~\cite{kuratowski,whyburn}). 
\begin{thm} \label{thm:simply-con}
Let $V$ be a domain in $\clC$. Then the following statements are 
equivalent.
\begin{itemize}
\item $V$ is simply connected,
\item $\clC\setminus V$ is connected 
\item $\bd V$ is connected.
\end{itemize}
\end{thm}

\begin{thm}
A path-connected continuum $X \subset \clC$ is simply connected if and only if 
it does not separate the plane.
\end{thm}

We will use the following theorems.

\begin{thm}[{\bf Mazurkiewicz--Moore Theorem} \cite{kuratowski}]
Let $X$ be a complete, locally connected metric space.
Then for every open connected set $V \subset X$ and every $x, y \in V$ there 
exists a topological arc in $V$ joining $x$ and $y$.
\end{thm}

\begin{thm}[\bf{Torhorst Theorem {\cite[Theorem 2.2]{whyburn}}}] 
\label{theorem:torhorst}
If $X$ is a locally connected continuum in $\clC$, then
the boundary of every component of $\clC \setminus X$ is a locally connected
continuum.
\end{thm}

%%%%%%%%%%%%%%%%%%%%%%%%%%%%%%%
%%%%%%%%%%%%%%%%%%%%%%%%%%%%%%%%
\section{Accesses and radial limits: proof of the Correspondence Theorem}\label{sec:accesses}

In this section we prove the Correspondence Theorem formulated in the introduction, which describes the relation between the topology of accesses and boundary behaviour of the Riemann map. To that end, we first state some preliminaries. Although some of them are a kind of folklore, we include their proofs for completeness.

For $z_1, z_2 \in \C$, we denote by $[z_1, z_2]$ the straight line segment joining $z_1$ to $z_2$. 

\begin{lem}\label{lem:no_boundary1} Let $U \subset\C$ be a simply connected domain, $z_0\in U$ and $v\in\partial U$. Let $\gamma_0, \gamma_1: [0,1] \to \clC$ be curves 
such that $\gamma_j([0,1)) \subset U$, $\gamma_j(0) = z_0$, $\gamma_j(1) = v$ for $j = 0,1$. Then the following statements hold.
\begin{enumerate}
\item[(a)] $\gamma_0$ and $\gamma_1$ are in the same access to $v$ if and only there is exactly one component of $\clC \setminus (\gamma_0 \cup \gamma_1)$ 
intersecting $\bd U$. 
\item[(b)] $\gamma_0$ and $\gamma_1$ are in different accesses to $v$ if and only 
if there are exactly two components of $\clC \setminus (\gamma_0 \cup \gamma_1)$ 
intersecting $\bd U$. 
\end{enumerate}

\end{lem}
\begin{proof} Note first that $\bd U$ is connected and $\gamma_0 \cup \gamma_1$ is connected and locally connected. Moreover, since $\bd U$ is infinite, there is at least one component of $\clC \setminus (\gamma_0 \cup \gamma_1)$ intersecting $\bd U$.

Now we show that there are at most two such components. On the contrary, suppose there are three different components $V_0, V_1, V_2$ of $\clC \setminus (\gamma_0 \cup \gamma_1)$ intersecting $\bd U$ and let $v_j \in \bd U \cap V_j$, $j = 0, 1, 2$. Note that $v \in \bd V_j$ for every $j$, because otherwise $\bd V_j \subset U$ for some $j$, which contradicts the connectivity of $\bd U$. Since $\gamma_0 \cup \gamma_1$ is connected, $V_j$ is simply connected. Let $\psi_j: \D \to V_j$ be a Riemann map such that $\psi_j(0) = v_j$. By Theorem~\ref{theorem:torhorst}, $\bd V_j$ is locally connected, so $\psi_j$ extends continuously to $\bdd$. As $v \in \bd V_j$, there exists $\xi_j \in \bd \D$ such that $\psi_j(\xi_j) = v$. 

Note that the curves $\psi_j([0, \xi_j])$ are topological arcs landing at the common endpoint $v$, such that $\psi_j([0, \xi_j]) \setminus \{v\}$ are disjoint curves in $\clC \setminus (\gamma_0 \cup \gamma_1)$. Hence, there is a neighbourhood $W$ of $v$, such that $W \setminus \left(\psi_0([0, \xi_0]) \cup \psi_1([0, \xi_1]) \cup \psi_2([0, \xi_2])\right)$ consists of three components $W_0, W_1, W_2$. Since $(\gamma_0([0, 1)) \cup \gamma_1([0, 1))) \cap (V_0 \cup V_1 \cup V_2) = \emptyset$ and $\gamma_0(1) = \gamma_1(1) = v$, each of the two ``tails'' $\gamma_0((t, 1))$, $\gamma_1((t, 1))$ for $t$ sufficiently close to $1$ is contained in one of the components $W_0, W_1, W_2$. Hence, one of them (say $W_0$) does not intersect any of the two tails. Then we can join two different curves $\psi_j([0, \xi_j])$ by a curve in $W_0 \setminus (\gamma_0 \cup \gamma_1)$, which is a contradiction.
Hence, there are at most two components of $\clC \setminus (\gamma_0 \cup \gamma_1)$ intersecting $\bd U$.

Now it is clear that the statements~(a) and~(b) will follow by showing that $\gamma_0$ and $\gamma_1$ are in the same access to $v$ if and only if there is exactly one component of $\clC \setminus (\gamma_0 \cup \gamma_1)$ intersecting $\bd U$.

In one direction we argue by contradiction. Assume that $\gamma_0$ and $\gamma_1$ are in the same access to $v\in\partial U$ and there are two different (simply connected) components $V_0, V_1$ of $\clC \setminus (\gamma_0 \cup \gamma_1)$ intersecting $\bd U$, with $v_j\in V_j\cap \partial U,\ j=0,1$. By a conformal change of coordinates in $\clC$, we may assume $v_1 = \infty$, which implies $\overline{V_0} \subset \C$.

Since $\gamma_0$ and $\gamma_1$ are in the same access to $v$, there exists a suitable homotopy between $\gamma_0$ and $\gamma_1$, i.e. a continuous map $H: [0,1]^2 \to U \cup \{v\}$ such that $H(t,0) = \gamma_0(t)$, $H(t,1) = \gamma_1(t)$ for $t \in [0,1]$, $H([0, 1)\times[0,1]) \subset U$ and $H(0,s) = z_0$, $H(1,s) = v$ for $s \in [0,1]$. Let $B \subset \C$ be a disc centered at $v$ of such a small radius that $v_0 \notin B$ and choose $t_0 \in (0, 1)$ such that
\begin{equation}\label{eq:HinB}
H\left([t_0,1] \times [0,1]\right) \subset B.
\end{equation}

As previously, let $\psi_0:\D \to V_0$ be the Riemann map which extends continuously to $\bdd$, such that $\psi_0(0)=v_0$ and $\psi_0(\xi_0)=v$ for some $\xi_0 \in \bdd$. Notice that $\xi_0$ is the unique point in $\bd \D$ such that $\psi_0(\xi_0) = v$. Indeed, suppose $\psi_0(\xi_0') = v$ for some $\xi_0' \in \bd \D$, 
$\xi_0' \neq \xi_0$. Then, by the Fatou Theorem, $\psi_0([0, \xi_0] \cup [0, \xi_0'])$ is a Jordan curve contained in $(\clC \setminus (\gamma_0 \cup \gamma_1)) \cup \{v\}$ separating the path-connected set $(\gamma_0 \cup \gamma_1) \setminus \{v\}$, which is impossible. Hence, the point $\xi_0$ is unique.

Take $\theta_0 \in \R$ such that $\xi_0 = e^{i\theta_0}$. We can assume $\theta_0 \in (0, 2\pi)$. By the continuity of $\psi_0$, there exist $\theta_1,\theta_2 \in(0,2\pi)$ with $0 <\theta_1 < \theta_0 < \theta_2 < 2\pi$ 
and $0 < \epsilon < 1$, such that 
\begin{equation}\label{eq:psiinB}
\psi_0(\{\rho e^{i \theta} \mid \rho \in [(1-\epsilon), 1], \; \theta \in 
[\theta_1,\theta_2]\}) \subset B.
\end{equation}
Note that $\psi_0(e^{i\theta_1}), \psi_0(e^{i\theta_2}) \neq v$ by the uniqueness of $\xi_0$. Hence, $\psi_0(e^{i\theta_j}) = \gamma_{k_j}(t_j)$, $j = 1, 2$, for some $k_1, k_2 \in \{0, 1\}$, and $t_1, t_2 \in (0,1)$. Moreover, by the continuity of $\psi_0$, we can choose $\theta_1, \theta_2$ such that
\begin{equation}\label{eq:>t_0}
t_1, t_2 \in [t_0,1).
\end{equation}

Let
\[
\alpha = \psi_0(\{\rho e^{i \theta_1} \mid \rho \in [\rho_0,1]\} \cup \{\rho_0 e^{i \theta} \mid \theta \in [0,2\pi] \setminus 
(\theta_1,\theta_2)\} \cup \{\rho e^{i \theta_2} \mid \rho \in [\rho_0,1]\}) 
\]
for $\rho_0 \in (1-\epsilon, 1)$. We claim that taking $\rho_0$ sufficiently close to $1$, we have
\begin{equation}\label{eq:alpha}
\alpha \subset U.
\end{equation}
Indeed, otherwise there exists a sequence $u_n \in \D$, such that $\psi_0(u_n) \in V_0 \setminus U$ and $u_n \to \xi_0'$ 
for some $\xi_0' \in \bd \D \setminus \{e^{i\theta} \mid \theta \in (\theta_1,\theta_2)\}$, which implies $\xi_0' \neq \xi_0$ and $\psi_0(\xi_0') \in \bd U \cap \bd V_0 = \{v\}$, making a contradiction with the uniqueness of $\xi_0$.  

By \eqref{eq:>t_0}, we can connect $\gamma_{k_1}(t_1)$ to $\gamma_{k_2}(t_2)$ by a curve $\beta \subset H\left([t_0, 1) \times [0,1]\right)$. Then it follows from \eqref{eq:HinB} that
\begin{equation}\label{eq:beta}
\beta \subset U \cap B.
\end{equation}
Let us define a closed (not necessarily Jordan) curve by
\[
\eta := \alpha \cup \beta.
\]
It follows from \eqref{eq:psiinB}, \eqref{eq:alpha} and \eqref{eq:beta} that 
\begin{equation}\label{eq:eta}
\eta \subset U \cap (V_0 \cup B).
\end{equation}
See Figure~\ref{fig:One_Two_Accesses}.

\begin{figure}[hbt!]
\centering
\def\svgwidth{0.7\textwidth}
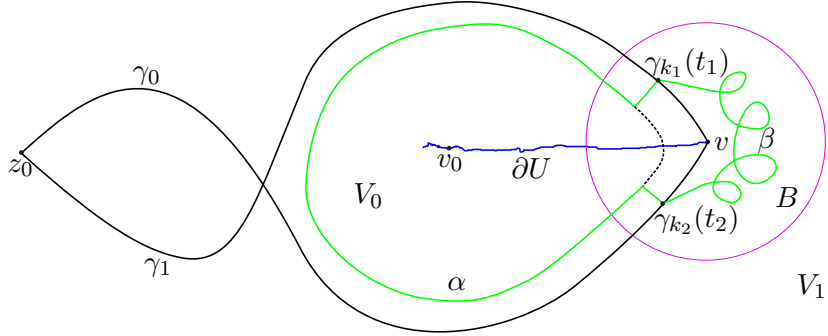
\caption{\small Construction of the curve $\eta$ in the proof of Lemma~\ref{lem:no_boundary1}.} 
\label{fig:One_Two_Accesses}
\end{figure}

Note that the index of $v_0$ with respect to the Jordan curve $\psi_0(\bd 
\D_{\rho_0})$ (with canonical orientation) is equal to $1$ and, by \eqref{eq:psiinB} and \eqref{eq:beta}, $\eta$ differs from $\psi_0(\bd \D_{\rho_0})$ by a closed curve contained in $B$. Hence, since $B$ is a simply connected set outside $v_0$, the index of $v_0$ with respect to $\eta$ (suitably oriented) is also equal to $1$. This implies that $v_0$ is in a bounded component of $\C \setminus \eta$, so $v_0$ and $v_1$ are separated in $\clC$ by $\eta$.  
Since $\eta \subset U$ by \eqref{eq:eta} and $v_0, v_1 \in \bd U$, this contradicts the connectivity of $\bd U$. We have proved that if $\gamma_0, \gamma_1$ are in the same access to $v$, then there is exactly 
one component of $\clC \setminus (\gamma_0 \cup \gamma_1)$ intersecting $\bd U$.

Conversely, suppose that there is exactly one component $V_0$ of $\clC \setminus (\gamma_0 \cup \gamma_1)$ intersecting $\bd U$. Let $K = \clC \setminus V_0$. By Theorem~\ref{thm:simply-con}, $K$ is simply connected. Moreover, the connectivity of $\bd U$ implies that $K$ is the union of $\gamma_0 \cup \gamma_1$ and the components of $\clC \setminus (\gamma_0 \cup \gamma_1)$ contained in $U$, in particular
\[
K \subset U \cup \{v\}.
\]
By the simply connectedness of $K$, the curves $\gamma_0$ and $\gamma_1$ are homotopic (with fixed endpoints) within $K \subset U \cup \{v\}$, which implies that they are in the same access to $v$.
\end{proof}

\begin{lem}\label{lem:arc} For every access $\A$ and every $\gamma \in \A$ there 
exists a topological arc $\gamma' \subset \gamma$ such that $\gamma' \in \A$. In particular, 
every access contains a topological arc.
\end{lem}
\begin{proof}
Let $\gamma \in \A$. Let $K$ be the union of $\gamma$ and all components of 
$\clC\setminus \gamma$ contained in $U$. Then $K \subset U$ and $K$ is simply 
connected because otherwise there would exist a loop in $K$ separating the 
boundary of $U$. By the Mazurkiewicz--Moore Theorem, $K$ is path-connected and there exists a topological arc $\gamma' \subset \gamma$ joining $z_0$ and $v$. Then $\gamma$ and $\gamma'$ are homotopic with fixed endpoints within $K$, so $\gamma' \in \A$. 
\end{proof}

\begin{proof}[Proof of the Correspondence Theorem]
To establish the one-to-one correspondence stated in the theorem we construct a bijective map $\Phi$ between accesses from $U$ to $v$ and points in $\partial \D$ for which the radial limit of $\varphi$ is equal to $v$.

Let $\A$ be an access to $v$ and take $\gamma_0 \in \A$. By the Lindel\"of Theorem, $\varphi^{-1}(\gamma_0)$ lands at a point $\zeta \in \bdd$ and the radial limit of $\varphi$ at $\zeta$ is equal to $v$. Set $\Phi(\A) = \zeta$. We claim that $\zeta$ does not depend on the choice of $\gamma_0$ (i.e.~the map $\Phi$ is well-defined). To see this, take some $\gamma_1 \in \A$ and suppose that $\varphi^{-1}(\gamma_1)$ lands at a point $\zeta' \in \bdd$, $\zeta' \neq \zeta$. By the Fatou Theorem, there 
exist $\xi_1, \xi_2 \in \bdd$ separating $\zeta$ and $\zeta'$ in $\bdd$, such 
that the radial limits of $\varphi$ at $\xi_1$ and $\xi_2$ exist and are 
different from $v$. Then, by construction, these radial limits are in two different components of $\clC \setminus (\gamma_0 \cup \gamma_1)$, which contradicts Lemma~\ref{lem:no_boundary1} since $\gamma_0$ and $\gamma_1$ are in the same access.

Now we prove the injectivity of $\Phi$. Suppose that two different accesses $\A_0$, $\A_1$ to $v$ correspond to the same point $\zeta \in\bdd$ and take $\gamma_0 \in \A_0$, $\gamma_1 \in \A_1$. By Lemma~\ref{lem:no_boundary1}, there exist two different components $V_0, V_1$ of $\clC \setminus (\gamma_0 \cup \gamma_1)$ intersecting $\bd U$. Since accessible points are dense in the boundary of a domain, there are points $v_0 \in \bd U \cap V_0$, $v_1 \in \bd U \cap V_1$ such that $v_0, v_1 \neq v$ and $v_j$ is accessible from $U \cap V_j$ by a curve $\eta_j$ for $j = 0, 1$. By the Lindel\"of Theorem, the curves $\varphi^{-1}(\eta_0), \varphi^{-1}(\eta_1)$ land at some points in $\bdd$ different from $\zeta$. But $\varphi^{-1}(\gamma_0), \varphi^{-1}(\gamma_1)$ land at $\zeta$, so $\varphi^{-1}(\eta_0)$ and $\varphi^{-1}(\eta_1)$ can be connected by an arc close to $\bdd$, disjoint from $\varphi^{-1}(\gamma_0) \cup \varphi^{-1}(\gamma_1)$. Hence, $\eta_0$ and $\eta_1$ can be connected by a curve in $U$ not intersecting $\gamma_0 \cup \gamma_1$, which 
is a contradiction.

Take now $\zeta \in \bdd$, such that the radial limit of $\varphi$ is equal to $v$. Then $\gamma := \varphi([0,\zeta])$ defines an access $\A$. By the Lindel\"of Theorem, $\Phi(\A) = \zeta$. It only remains to see that it we take any $\eta \subset \D$ landing at $\zeta$, such that $\varphi(\eta)$ lands at some point $w \in\clC$, then $w=v$ and $\varphi(\eta) \in \A$. The first fact holds by the Lindel\"of Theorem. Knowing $w=v$, the fact that $\varphi(\eta) \in \A$ follows from the injectivity of $\Phi$. 

Finally, the suitable correspondence between accesses to $v$ and the prime ends of $\varphi$ follows directly from the general classification of the prime ends of $\varphi$ and their correspondence to points of $\bd \D$, see e.g.~\cite[Theorem~9.4]{col-loh} for details.
\end{proof}

%%%%%%%%%%%%%%%%%%%%%%%%%%%%%
%%%%%%%%%%%%%%%%%%%%%%%%%%%%%
%%%%%%%%%%%%%%%%%%%%%%%%%%%%%
%%%%%%%%%%%%%%%%%%%%%%%%%%%%%
\section{Proofs of Theorems A and B}\label{sec:proofs}

Our goal in this section is to prove a dynamical version of the Correspondence Theorem (Theorem~A) together with an application concerning accesses to infinity and weakly repelling fixed points of meromorphic maps (Theorem~B). 

By $\varrho_V(\cdot)$, $\varrho_V(\cdot,\cdot)$ we will denote, respectively, the density of the hyperbolic metric and the hyperbolic distance in a simply connected hyperbolic domain $V \subset \C$. We will use the following standard estimate of the hyperbolic metric in simply connected domains:
\begin{equation} \label{eq:hyperb}
\frac{1}{2\dist(z, \bd V)} \leq \varrho_V(z) \leq \frac{2}{\dist(z, \bd V)} \qquad \text{for } z \in V,
\end{equation}
where $\dist$ is the Euclidean distance on the plane (see e.g.~\cite[Theorem 4.3]{carlesongamelin}).

The proof of Theorem~A is based on the three following lemmas. 
As a preliminary, we show a useful condition to ensure that two curves land at a boundary point through the same access. 

\begin{lem} \label{lem:hyp}
Let $U\subset \C$ be a simply connected domain and let $z_0 \in U$, $v \in \bd U$. Suppose that $\gamma_0, \gamma_1: [0,1) \to U$ are curves such that $\gamma_0(0) = \gamma_1(0) = z_0$ and $\gamma_0$ lands at a point $v \in \bd U$. If there exists $c > 0$ such that $\varrho_U(\gamma_0(t), \gamma_1(t)) < c$ for every $t \in [0, 1)$, then $\gamma_1$ lands at $v$ and $\gamma_0$, $\gamma_1$ are in the same access to $v$ from $U$. 
\end{lem}
\begin{proof} By a conformal change of coordinates, we can assume $v \in \C$. Then by \eqref{eq:hyperb}, we have
\[
\varrho_U(z) \geq \frac{1}{2|z-v|} 
\]
for $z \in U$, which implies
\[
c > \varrho_U(\gamma_0(t), \gamma_1(t)) \ge \frac 1 2 \int_{|\gamma_0(t) - v|}^{|\gamma_1(t) - v|} \frac{dx}{x} = \ln\left|\frac{\gamma_1(t) - v}{\gamma_0(t) - v}\right|.
\]
Since $|\gamma_0(t) - v| \to 0$ as $t \to 1^-$, we must have $|\gamma_1(t) - v| \to 0$ as $t \to 1^-$, which means that $\gamma_1$ lands at $v$.

To show that $\gamma_0$, $\gamma_1$ are in the same access to $v$, let $\alpha_t : [0,1] \to U$, $t \in [0,1)$, be the hyperbolic geodesic joining $\gamma_0(t)$ to $\gamma_1(t)$ in $U$. Define $H(t, s) = \alpha_t(s)$ for $(t,s) \in [0,1) \times [0,1]$ and $H(1, s) = v$ for $s \in [0,1]$. To check that $H$ is a suitable homotopy between $\gamma_0$ and $\gamma_1$, it is enough to show that $\alpha_{t_n}(s_n) \to v$ for every $t_n \in [0, 1)$, $t_n \to 1^-$, $s_n \in [0,1]$. Suppose otherwise. Then, choosing a subsequence, we can assume that there exists $d > 0$ such that $|\alpha_{t_n}(s_n) - v| > d$ for every $n$. Again by \eqref{eq:hyperb},
\[
 \varrho_U(\gamma_0(t_n), \alpha_{t_n}(s_n)) \ge \frac 1 2 \int_{|\gamma_0(t_n) - v|}^{|\alpha_{t_n}(s_n) - v|} \frac{dx}{x} = \ln\left|\frac{\alpha_{t_n}(s_n) - v}{\gamma_0(t_n) - v}\right| > \ln\frac{d}{|\gamma_0(t_n) - v|}.
\]
Since $|\gamma_0(t_n) - v| \to 0$ as $n \to \infty$ and $\alpha_{t_n}$ is the geodesic, we have 
\[
c > \varrho_U(\gamma_0(t_n), \gamma_1(t_n)) \geq \varrho_U(\gamma_0(t_n), \alpha_{t_n}(s_n)) \to \infty,
\]
which is a contradiction. Hence, $\gamma_0$, $\gamma_1$ are in the same access to $v$.
\end{proof}

The next lemma studies the local behaviour of a map near a regular boundary  fixed point.

\begin{lem}\label{lem:angular} Let $g: \D \to \D$ be a holomorphic map and $\zeta \in \bdd$ a regular boundary  fixed point of $g$. Then there exists $c > 0$ such that $\varrho_\D(g(t\zeta), t\zeta) < c$ for every $t \in [0, 1)$.
\end{lem}

\begin{proof}
By the Julia--Wolff Lemma, there exists $a \in (0,\infty)$ such that
\[
\frac{g(t\zeta)- \zeta}{(t-1)\zeta} \to a
\]
as $t \to 1^-$. Hence,
\[
g(t\zeta) - t\zeta = (a - 1 + h(t))(t-1)\zeta
\]
with $|h(t)| \to 0$ as $t\to 1^-$, so for $t \in (t_0, 1)$, with some $t_0$ close to $1$, we have 
\[
\varrho_\D(g(t\zeta),t\zeta) \leq \frac{2|g(t\zeta) - t\zeta|}{\inf_{u \in [g(t\zeta),t\zeta]}\dist(u, \bdd)} <  \frac{3|g(t\zeta) - t\zeta|}{\min(a, 1) (1-t)} < \frac{3|a - 1 + h(z)|}{\min(a, 1)} < \frac{3|a - 1| + 1}{\min(a, 1)}.
\]
Since for $t \in [0, t_0]$ the estimation is obvious, the lemma is proved.
\end{proof}

The third lemma is a modification of the classical Douady--Hubbard--Sullivan Snail Lemma (see e.g.~\cite[Lemma 16.2]{milnor}) on the multiplier of the landing point of an invariant curve in a Fatou component $U$. The difference with a standard setup is that instead of invariance, we assume that the hyperbolic distance in $U$ between a point of the curve and its image under the map is bounded.  

\begin{lem}[{\bf Modified Snail Lemma}{}]\label{prop:lands}
Let $f:\C \to \clC$ be a meromorphic map and $U$ a simply connected invariant Fatou component of $f$. Suppose that a curve $\gamma: [0, 1) \to U$ lands at a fixed point $v\in \bd U$ of $f$ and there exists $c > 0$ such that $\varrho_U(f(\gamma(t)), \gamma(t)) < c$ for every $t \in [0,1)$. Then $v$ is a weakly repelling fixed point of $f$. 
\end{lem}

\begin{proof}By a conformal change of coordinates, we can assume $v \in \C$.
We need to show that $|f'(v)| > 1$ or $f'(v) = 1$. Suppose this does not hold. If $|f'(v)| < 1$, then $v$ is in the Fatou set of $f$, which contradicts the fact $v \in \bd U$. Hence, we are left with the case 
\[
f'(v) = e^{2 \pi i \beta}, \quad\beta \in (0, 1). 
\]
Since $\beta \in (0, 1)$, we can take an arbitrarily small positive $\varepsilon < \min(\beta, 1 -\beta)$ and positive integers $p, q$ such that 
\begin{equation}\label{eq:qp}
|q\beta - p| < \varepsilon.
\end{equation}
This can be accomplished by using, for example, a sequence of rational numbers $p_n/q_n$ approximating $\beta$ with rate $1/q_n^2$.

The strategy of the proof is as follows. Using the asymptotics of $f$ near $v$, i.e.~a rotation of angle $\beta$, close to $p/q$, we construct a closed curve $\kappa \subset U$ near $v$ and prove that the index of $\kappa$ with respect to $v$ is non-zero, contradicting the simply connectivity of $U$. We now proceed to make this idea precise.

Let $w = w(t) \in \bd U$ be such that $\dist(\gamma(t), \bd U) = |\gamma(t) - w|$. By \eqref{eq:hyperb}, we have
\[
\varrho_U(z) \geq \frac{1}{2|z-w|} 
\]
for $z \in U$, which gives
\[
c > \varrho_U(f(\gamma(t)), \gamma(t)) \ge \frac 1 2 \int^{|f(\gamma(t)) - w|}_{|\gamma(t) - w|} \frac{dx}{x} = \ln\left|\frac{f(\gamma(t)) - w}{\gamma(t) - w}\right|.
\]
Hence,
\[
e^c > \frac{|f(\gamma(t)) - w|}{|\gamma(t) - w|} \ge \frac{|f(\gamma(t)) - \gamma(t)| - |\gamma(t) -w|}{|\gamma(t) - w|} = \frac{|f(\gamma(t)) - \gamma(t)|}{\dist(\gamma(t), \bd U)} - 1. 
\]
This implies that there exists $\delta > 0$ (any $0<\delta < \frac{1}{e^c+1}$ would suffice) such that 
\begin{equation}\label{eq:delta}
\D(\gamma(t), \delta |f(\gamma(t)) - \gamma(t)|) \subset U.
\end{equation}
for every $t \in [0,1)$. 

Fix some $t \in [0,1)$ close to $1$ and let $\eta_0$ be the hyperbolic geodesic connecting $\gamma(t)$ to $f(\gamma(t))$ in $U$. We claim that $\eta_0$ is contained in an arbitrarily small neighbourhood of $v$, if $t$ is close enough to $1$. Indeed, if $u \in \eta_0$ is such that $\max_{z \in \eta_0} |z-v| = |u - v|$, then using \eqref{eq:hyperb} and repeating the above estimates we obtain
\[
c > \varrho_U(f(\gamma(t)), \gamma(t)) \ge \varrho_U(u, \gamma(t)) \ge \frac 1 2 \int_{|\gamma(t)) - v|}^{|u-v|} \frac{dx}{x} = \ln\left|\frac{u - v}{\gamma(t) - v}\right|.
\]
Hence, since $|\gamma(t) - v| \to 0$ as $t \to 1^-$, the maximal distance from points of $\eta_0$ to $v$ is arbitrarily small if $t$ is close enough to $1$.

Let
\[
\eta = \bigcup_{j= 0}^{q-1} \eta_j, \qquad \text{where} \quad \eta_j = f^j(\eta_0) \quad \text{for } j = 1, \ldots, q-1. 
\]
Then
\begin{equation}\label{eq:inU}
\eta \subset U 
\end{equation}
and $\eta$ is contained in a small neighbourhood of $v$, if $t$ is close enough to 1.
Denote by $\Delta_{\Arg}$ the growth of the argument function $z \mapsto \frac{1}{2\pi}\Arg(z - v)$ along a curve. Since $f(z) - v = e^{2 \pi i \beta}(z - v) + O((z - v)^2)$ for $z$ close to $v$, 
we have
\[
|\Delta_{\Arg} (\eta_0) - \beta - m| < \frac 1 {2q} 
\]
for some $m \in \Z$, if $t$ is sufficiently close to $1$. Similarly, since $\eta_j$ are consecutive images of $\eta_0$ under $f$, we have
\[
|\Delta_{\Arg} (\eta_j) - \beta - m| < \frac 1 {2q} 
\]
for $j = 0, \ldots, q-1$. Summing up, we have
\[
|\Delta_{\Arg} (\eta) - q \beta - qm| < \frac 1 2,
\]
so by \eqref{eq:qp},
\begin{equation}\label{eq:arg}
|\Delta_{\Arg} (\eta) - p - qm| < |\Delta_{\Arg} (\eta) -q\beta -q m|+|q\beta -p|   <\frac 1 2 + \varepsilon.
\end{equation}

Note that 
\begin{equation}\label{eq:not0}
p+qm \neq 0,
\end{equation}
because otherwise by \eqref{eq:qp}, we would have $|\beta + m | < \varepsilon/q \leq \varepsilon$, which contradicts the conditions $\beta \in (0, 1)$,  $\varepsilon < \min(\beta, 1 -\beta)$.

We have $f^q(z) - v = e^{2 \pi i q\beta}(z - v) + O((z - v)^2)$ for $z$ near $v$,
which together with \eqref{eq:qp} implies
\[
\frac{|f^q(\gamma(t)) - \gamma(t)|}{|\gamma(t) - v|} \to |1 - e^{2\pi i q\beta}| < |1-e^{2\pi i\varepsilon}| <  2\pi \varepsilon, 
\]
so
\begin{equation}\label{eq:arg2}
\Delta_{\Arg}([f^q(\gamma(t)), \gamma(t)]) < 2\pi \varepsilon.
\end{equation}
Moreover,
\[
\frac{|f^q(\gamma(t)) - \gamma(t)|}{|f(\gamma(t)) - \gamma(t)|} \to \frac{|1 - e^{2\pi i q\beta}|}{|1 - e^{2\pi i \beta}|} < \frac{2\pi \varepsilon}{|1 - e^{2\pi i \beta}|} < \delta,
\]
as $t \to 1^-$, if $\varepsilon$ is taken sufficiently small. Hence, by \eqref{eq:delta},
\begin{equation}\label{eq:inU2}
[f^q(\gamma(t)),\gamma(t)] \subset U
\end{equation}
for $t$ sufficiently close to $1$. 

Concluding, by \eqref{eq:inU}, \eqref{eq:inU2}, \eqref{eq:arg} and \eqref{eq:arg2}, the curve 
\[
\kappa = \eta \cup [f^q(\gamma(t)), \gamma(t)]
\]
is a closed curve in $U$ and
\[
|\Delta_{\Arg}(\kappa) -p - qm| < \frac 1 2 + (2\pi + 1) \varepsilon < 1,
\]
provided $\varepsilon$ is taken sufficiently small and $t$ is sufficiently close to $1$. Hence, in fact 
\[
\Delta_{\Arg}(\kappa) = p + qm.
\]
By \eqref{eq:not0}, we conclude that the index of $v$ with respect to $\kappa$ is non-zero. This means that $\tilde \eta$ is not contractible in $U$, which contradicts the simple connectedness of $U$ and concludes the proof.
\end{proof}

Now we are ready to prove Theorem~A.

\begin{proof}[Proof of Theorem \rm A] 
To prove (a), let $\A$ be an invariant accesses to $v$ and let $\gamma \in \A$, such that $f(\gamma) \cup \eta \in \A$, where $\eta$ is a curve connecting $z_0$ to $f(z_0)$ in $U$. We have $\gamma(t) \to v$, $f(\gamma(t)) \to v$ as $t \to 1^-$, so by continuity, either $v = \infty$ or $f(v) = v$.

By the Correspondence Theorem, the curves $\varphi^{-1}(\gamma)$ and $g(\varphi^{-1}(\gamma)) = \varphi^{-1}(f(\gamma))$ land at a common point $\zeta \in \bdd$, such that the radial limit of $\varphi$ at $\zeta$ is equal to $v$. By the Lehto--Virtanen Theorem, the radial limit of $g$ at $\zeta$ exists and is equal to $\zeta$, so $\zeta\in\bdd$ is a boundary fixed point of $g$.

To prove (b), let $\zeta\in\bdd$ be a regular boundary fixed point of $g$ and let \[
\gamma(t) = \varphi(t \zeta)                                                             
\]
for $t \in [0, 1)$. By Lemma~\ref{lem:angular}, there exists $c > 0$ such that
\begin{equation}\label{eq:<c}
\varrho_U(f(\gamma(t)), \gamma(t)) = \varrho_\D(g(t\zeta), t\zeta) < c.
\end{equation}

Suppose that $\gamma(t_n) \to v$ for some point $v \in \bd U \cap \C$ and a sequence $t_n \in [0,1)$, $t_n \to 1^-$. Using \eqref{eq:<c} and \eqref{eq:hyperb}, we obtain
\[
c > \varrho_U(f(\gamma(t_n)), \gamma(t_n)) \geq  \frac 1 2 \int_{|\gamma(t_n) - v|}^{|f(\gamma(t_n)) - v|} \frac{dx}{x} = \ln\left|\frac{f(\gamma(t_n)) - v}{\gamma(t_n) - v}\right|.
\]
Since $|\gamma(t_n) - v| \to 0$ as $n \to \infty$, we have $|f(\gamma(t_n)) - v| \to 0$, so by continuity, $f(v) = v$. Hence, the limit set of the curve $\gamma(t)$ for $t \to 1^-$ is contained in the set of fixed points of $f$ in $\bd U\cap \C$ together with $\infty$. Since it is a discrete set, in fact $\gamma(t) \to v$ as $t \to 1^-$, where $v\in \bd U \cap \C$ is a fixed point of $f$. 

By \eqref{eq:<c}, the assumptions of Lemma~\ref{prop:lands} are satisfied for the curve $\gamma$, so the lemma implies that $v$ is weakly repelling. 

Let $\eta$ be a curve connecting $\gamma(0)$ to $f(\gamma(0))$ in $U$. The fact that $\gamma$ and $f(\gamma) \cup \eta$ are in the same access $\A$ follows directly from \eqref{eq:<c} and Lemma~\ref{lem:hyp}. Hence, by the Correspondence Theorem, 
$\zeta$ corresponds to an invariant access to $v$.

\end{proof}

\begin{proof}[Proof of Theorem \rm B] Consider first the case $d < \infty$. Let $g$ be the inner function associated to $f|_U$. Recall that $g$ extends by reflection to the Riemann sphere as a finite Blaschke product of degree $d$. Hence, $g$ has finite derivative at all its boundary fixed points, so all of them are regular. Therefore, by Theorem~A and the Correspondence Theorem, $U$ has exactly $D$ invariant accesses to infinity or to weakly repelling points of $f$ in $\bd U \cap \C$ and there are no other invariant accesses from $U$ to its boundary points. 

To end the proof of the statement~(a), we estimate the number $D$.
As a rational map, $g$ has $d + 1$ fixed points in $\clC$, counted with multiplicities. Note that by the Denjoy--Wolff Theorem and Theorem~\ref{DWII}, if $d > 1$, then we can have either none or two fixed points in $\clC \setminus \bdd$ (and they are attracting) and on $\bdd$ there can be at most one non-repelling (attracting or parabolic) fixed point. In fact, the map $g$ on $\D$ is of one of the following four types: elliptic, hyperbolic, simply parabolic or doubly parabolic (see e.g.~\cite{absorb,faghen}). In the elliptic case, all fixed points of $g$ are non-degenerate (i.e.~their multipliers are not equal to $1$), there are two attracting fixed points outside $\bdd$ (the Denjoy--Wolff points for $g|_\D$ and $g|_{\clC\setminus\overline\D}$) and all other fixed points are repelling, so $D = d - 1$. In the hyperbolic case, all 
fixed points of $g$ are non-degenerate fixed points in $\bdd$ (one attracting and all others repelling), so $D = d + 1$. In the two parabolic cases, all fixed points of $g$ are in $\bdd$ and exactly one of them (which is the Denjoy--Wolff point for $g|_\D$) is degenerate (parabolic) of multiplicity two or three, respectively. Hence, $D = d$ or $d - 1$ in this case. For details, see e.g.~\cite{faghen}. 

In this way we have proved that $d -1 \le D \le d +1$. In particular, this implies that $D \ge 1$ in the case $d > 1$. If $d = 1$, then the only possible cases are elliptic, simply parabolic or hyperbolic. In the two latter cases, $g$ has a fixed point in $\bdd$, so $D =1$. In the elliptic case, $f$ has an attracting or neutral fixed point in $U$. If the point is attracting, then $U$ contains a critical point, which contradicts $d = 1$. In the remaining case, the neutral fixed point is the center of an invariant Siegel disc. This shows that $D \geq 1$ unless $f|_U$ is a Siegel disc. This ends the proof of the statement~(a).

Suppose now that $d =\infty$ and $f|_U$ is singularly nice. Then there exists a singularity $\zeta^* \in\bdd$ of the inner function $g$ associated to $f|_U$ and an open arc $A \subset \bdd$, such that $\zeta^* \in A$ and for every $\zeta \in A \setminus \{\zeta^*\}$, the map $g$ has a finite angular derivative at $\zeta$. Hence, by Theorem~\ref{thm:bargmann}~part~(a), the map $g$ has infinitely many regular boundary  fixed points in $A$. By Theorem~A and the Correspondence Theorem, the union of the set of invariant accesses to $\infty$ from $U$ and the set of invariant accesses to weakly repelling fixed points of $f$ from $U$ is infinite, which immediately implies the statement~(b).

To show (c), suppose that $U$ is bounded. Then, by Proposition~\ref{degree}, we have $d < \infty$ and the claim follows immediately from the statement~(a).
\end{proof}

%%%%%%%%%%%%%%%%%%%%%%%%%%%%%%%%%
%%%%%%%%%%%%%%%%%%%%%%%%%%%%%%%%%
%%%%%%%%%%%%%%%%%%%%%%%%%%%%%%%%%
%%%%%%%%%%%%%%%%%%%%%%%%%%%%%%%%%
\section{Proofs of Theorems~D and~E and remaining results} \label{sec:TheoremE}

\begin{proof}[Proof of Theorem \rm D] Take $v \in \bd U$ and suppose that there are two different accesses to $v$ from $U$. By Lemma~\ref{lem:arc}, they are represented  topological arcs $\gamma_0, \gamma_1$ in $U$, disjoint up to the common landing point at $v$. Then $\gamma_0 \cup \gamma_1$ is a Jordan curve in $U \cup \{v\}$ and by Lemma~\ref{lem:no_boundary1}, both components of $\clC \setminus (\gamma_0 \cup \gamma_1)$ contain points from $\bd U \subset \mathcal J(N)$. This is a contradiction, since $V$ is contained in one component of $\clC \setminus (\gamma_0 \cup \gamma_1)$ and the complete invariance of $V$ implies $\mathcal J(N) = \bd V$. Hence, $U$ has at most one access to $v$.

Suppose now $d < \infty$. By Corollary~C, there are at least $\max(d -1, 1)$ invariant accesses to $\infty$ in $U$. Hence, in fact there is a unique access $\A$ to $\infty$ in $U$, it is invariant and $d \in \{1, 2\}$. In particular, this proves~(a).

Consider a Riemann map $\varphi: \D \to U$ and the inner function $g = \varphi^{-1} \circ N \circ \varphi: \D \to \D$ associated to $N_U$. By Theorem~A, the access $\A$ corresponds to a boundary fixed point $\zeta\in\bdd$. Since $g$ has finite degree, it extends to the Riemann sphere and has a derivative at $\zeta$, where $g'(\zeta) \in (0, +\infty)$.

Suppose that $d = 1$ and there is a pole $p \in \bd U \cap \C$ of $N$ accessible from $U$. Take a curve $\gamma_1$ in $U$ landing at $p$. Then $N(\gamma_1)$ is a curve in $U$ landing at $\infty$. Since $\A$ is the unique access to $\infty$ in $U$, the Correspondence Theorem implies that the curve $\varphi^{-1}(N(\gamma_1)))$ lands at $\zeta$ and the curve $\varphi^{-1}(\gamma_1)$ lands at some point $\zeta' \neq \zeta$ in $\bdd$. We can assume that $\varphi^{-1}(N(\gamma_1)))$ is contained in a small neighbourhood of $\zeta$. As $\deg g = 1$ and $g(\zeta) = \zeta$, the set $g^{-1}(\varphi^{-1}(N(\gamma_1)))$ is a curve landing at $\zeta$. This is a contradiction, since $\varphi^{-1}(\gamma_1)\subset g^{-1}(\varphi^{-1}(N(\gamma_1)))$ lands at $\zeta' \neq \zeta$. In this way we have showed the statement~(b).

To show the statement~(c), suppose now $d = 2$. Take a curve $\gamma$ in $U$ landing at infinity. Again, by the uniqueness of $\A$, the curve $\varphi^{-1}(\gamma)$ lands at $\zeta$ and we can assume that it is contained in a small neighbourhood of $\zeta$. 
Since the local degree of $g$ near $\zeta$ is $1$, the set $g^{-1}(\varphi^{-1}(\gamma))$ consists of two curves $\eta_0$ and $\eta_1$, such that $\eta_0$ lands at $\zeta$ and $\eta_1$ lands at some point $\zeta_1 \neq \zeta$ in $\bdd$. As $\varphi(\eta_1)$ is the preimage of $\gamma$ under $N$, it lands at $\infty$ or a pole of $N$ in $\bd U \cap \C$. The first case is impossible by the Correspondence Theorem, since $\A$ is the unique access to $\infty$ in $U$. This shows that $\bd U$ contains a pole accessible from $U$. To show that there are no other accessible poles in $\bd U$, we use arguments similar to those in the proof of the statement~(b).
\end{proof}

\begin{proof}[Proof of Proposition~\rm\ref{strongly}]
In one direction, the implication is trivial -- if $\A$ is strongly invariant, then for every $\gamma \in \A$, the curve $f(\gamma)$ lands at $v$. Suppose now that $\A$ is invariant and for every $\gamma \in \A$, the curve $f(\gamma)$ lands at some boundary point of $U$. Let $\varphi: \D \to U$ be a Riemann map and $g = \varphi^{-1} \circ f \circ \varphi: \D \to \D$ the inner function associated to $f_U$. By Theorem~A, the access $\cala$ corresponds to a boundary fixed point $\zeta\in\bdd$ of $g$. Suppose that for some $\gamma \in \cala$ the curve $f(\gamma)$ lands at a point of $\bd U$, such that $f(\gamma) \in \A'$ with $\cala'\neq \cala$. By the Correspondence Theorem, $\cala'$ corresponds to a point $\zeta'\in\bdd$ with $\zeta'\neq \zeta$. Observe that $\varphi^{-1}(\gamma)$ lands at $\zeta$ while $g(\varphi^{-1}(\gamma)) = \varphi^{-1}(f(\gamma))$ lands at $\zeta'$. Hence, by the Lehto--Virtanen Theorem, the radial limit of $g$ at $\zeta$ is equal to $\zeta'$, which is a contradiction with the fact that 
$\zeta$ is 
a boundary fixed point of $g$. Therefore, for every $\gamma \in \A$, we have $f(\gamma) \in \A$.  
\end{proof}

\begin{proof}[Proof of Proposition~\rm\ref{prop:sing}]
Let $\varphi: \D \to U$ be a Riemann map and $g = \varphi^{-1} \circ f \circ \varphi: \D \to \D$ the inner function associated to $f_U$. By Theorem~\ref{thm:bargmann}~part~(c), there exists $z \in W$, such that $\Sing(g)$ is the set of accumulation points of $g^{-1}(\varphi^{-1}(z))$. Since $f^{-1}(z) \cap U$ is contained in a finite number of curves landing at points of $\bd U$, by the Lindel\"of Theorem, the set $g^{-1}(\varphi^{-1}(z)) = \varphi^{-1}(f^{-1}(z) \cap U)$ is contained in a finite number of curves landing at points of $\bdd$. Hence, the set $\Sing(g)$ is finite, so every singularity of $g$ is isolated. On the other hand, since $\deg f|_U = \infty$, we have $\Sing(g) \neq \emptyset$. Therefore, $g$ has an isolated singularity in $\bdd$, so $f|_U$ is singularly nice.
\end{proof}

A part of Theorem~E is based on Proposition~\ref{poles}. 

\begin{proof}[Proof of Proposition \rm \ref{poles}]
We argue as in the proof of \cite[Lemma 11]{baker-dominguez}. 
Since $\infty$ is accessible from $U$, by the Lindel\"of Theorem there exists $\zeta\in\bdd$, such that the radial limit of $\varphi$ at $\zeta$ is equal to $\infty$.
Take $\zeta^* \in \Sing(g)$ and suppose $\zeta^* \notin \Acc(\Theta)$. Then $\zeta^* \in A$, where $A$ is an arc in $\bd \D$ such that $(A \setminus \{\zeta^*\}) \cap \Theta = \emptyset$. By Theorem~\ref{thm:bargmann}~part~(b), 
there exists $\zeta' \in A \setminus \{\zeta^*\}$ and a curve $\eta \subset \D$ landing at $\zeta'$, such that $g(\eta)$ lands at $\zeta$ in a Stolz angle. Since the radial limit of $\varphi$ at $\zeta$ is equal to infinity, it follows that $\varphi(g(\eta))=f(\varphi(\eta))$ lands at $\infty$.  Hence, any limit point of the curve $\varphi(\eta)$ is infinity or a pole of $f$. Since the limit set of $\varphi(\eta)$ is a continuum, in fact $\varphi(\eta)$ lands at infinity or a pole of $f$. By the Correspondence Theorem, this implies $\zeta \in \Theta$, which is a contradiction.
\end{proof}

We are now ready to prove Theorem~E.

\begin{proof}[Proof of Theorem \rm E] 
To show (a), assume $1 < d<\infty$ and suppose that $\bd U$ does not contain poles of $f$ which are accessible from $U$. Since infinity is accessible from $U$, there exists a curve $\gamma$ in $U$ landing at infinity. By the Lindel\"of Theorem, $\eta=\varphi^{-1}(\gamma)$ is a curve in $\D$ landing at a point $\zeta \in \bdd$.
Recall that $g$ extends to a finite Blaschke product of degree $d$ on $\chat$. 
Since $g(\D) \subset \D$ and $g(\clC \setminus \overline \D) \subset \clC \setminus \overline \D$, the Julia set $\mathcal J(g)$ is contained in $\bd \D$. Moreover, $\overline{\bigcup_{n= 0}^\infty g^{-n}(\zeta)} \supset \mathcal J(g)$ (see \cite[Theorem III.1.5]{carlesongamelin}). It follows that $\bigcup_{n= 0}^\infty g^{-n}(\zeta)$ is an infinite set in $\bdd$ and we can take an infinite sequence of disjoint points $\zeta_n \in g^{-n}(\zeta)$, $n = 1, 2,\ldots$. Then, taking a suitable inverse branch of $g^{-n}$ defined near $\zeta$, we construct a curve $\eta_n$ in $\D$ landing at $\zeta_n$ such that $g^n(\eta_n)=\eta$. We claim that $\gamma_n:=\varphi (\eta_n)$ lands at infinity for all $n\geq 1$. To show this, observe that $f(\gamma_1)=\varphi(g(\eta_1))=\varphi(\eta)=\gamma$ lands at infinity. Hence, $\gamma_1$ must land at infinity or at a pole of $f$ accessible from $U$. But the latter is impossible by assumption, thus $\gamma_1$ lands at infinity. Repeating the arguments by induction, we prove the claim.

We have showed that there are 
infinitely many curves $\eta_n$ in $\D$, such that $\eta_n$ lands at a point $\zeta_n \in \bdd$, the points $\zeta_n$ are disjoint and $\varphi(\eta_n)$ is a curve in $U$ landing at infinity. By the Correspondence Theorem, this defines infinitely many different accesses to infinity from $U$. From Theorem~A we know that at most $d+1$ of them are invariant.

To prove (b), assume $d=\infty$ and suppose that $\bd U$ contains only finitely many  poles of $f$ accessible from $U$. Then $\Sing(g)$ is non-empty so by Proposition~\ref{poles}, the set $\Theta$ is infinite. By the Correspondence Theorem, there are infinitely many accesses from $U$ to infinity or poles of $f$ in $\bd U$. 

Suppose that there are only finitely many accesses to infinity from $U$. Then there are infinitely many accesses to poles of $f$. Since by assumption, there are only finitely many poles accessible from $U$, there exists a pole $p$ with infinitely many accesses from $U$. Note that $f$ near a pole $p$ is a finite degree covering, possibly branched at $p$. Therefore, by Lemma~\ref{lem:no_boundary1}, if $p$ has infinitely many accesses from $U$, then $\infty$ has infinitely many accesses from $U$, which is a contradiction. Hence, there are infinitely many accesses to infinity from $U$. 
\end{proof}

%%%%%%%%%%%%%%%%%%%%
%%%%%%%%%%%%%%%%%%%%

\section{Examples} \label{sec:examples}

In this section we present a number of examples which illustrate our results. We start by a simple, well-know example which, however, exhibits many phenomena described in previous sections.

\begin{ex} \label{ex1}
Let  $f(z)=z+\tan z$. Then the following hold.
\begin{enumerate}[(a)]
\item The Fatou set $\mathcal{F}(f)$ consists of two completely invariant Baker domains 
\[
\H_+:=\{z\in\C\mid \Im(z)>0\}, \qquad \H_-:=\{z\in\C\mid \Im(z)<0\},
\]
in particular, $\deg f|_{\H_\pm} = \infty$.
\item The Julia set $\mathcal{J}(f)$ is equal to $\R \cup \{\infty\}$.
\item Each Baker domain $\H_\pm$ has a unique access to infinity, which is invariant but not strongly invariant.
\item $\bd \H_\pm$ contains infinitely many weakly repelling fixed points and infinitely many accessible poles of $f$.
\item The inner function associated to $f|_{\H_\pm}$ has exactly one singularity in $\partial \D$, so $f|_{\H_\pm}$ is singularly nice. 

\end{enumerate}
\end{ex}

\begin{proof}
Given the expression
\begin{equation}\label{eq:tan}
\tan (x+i y)= \frac{\sin 2x}{\cos 2x + \cosh 2y} + i \left( \frac{\sinh 2y}{\cos 2x + \cosh 2y}\right)
\end{equation}
for $x,y\in \R$, it is straightforward to check that the map is symmetric with respect to $\R$, $f(\R) = \R \cup \{\infty\}$ and $f(\H_\pm) \subset \H_\pm$, so $\H_\pm$ is contained in a Fatou component of $f$. Observe also that $z\mapsto \tan z$ has two asymptotic values $\pm i$ with asymptotic tracts in $\H_\pm$, respectively, and
\[
f(z) \simeq z\pm i  \quad \text{if} \quad  \Im (z) \to \pm\infty.
\]
This shows that $\H_\pm$ are in different Fatou components, which gives~(a) and~(b). 
Since $\bd\H_\pm \cap \C = \R$ is connected, it has a unique access to $\infty$ (cf.~Lemma~\ref{lem:no_boundary1}). Clearly, it is the dynamical access defined by the invariant curve $\gamma(t):=\pm i t$ for $t\geq 1$, and it is invariant. To see that it is not strongly invariant, observe that a small neighbourhood $V_k = V_0 + k\pi$ of the pole $p_k:=\pi/2 + k \pi$, $k\in\Z$ is mapped under $f$ onto a neighbourhood $W$ of infinity. Hence, given a point $z^* \in \H_\pm$ with large imaginary part, there exists a sequence of preimages $z_k\in f^{-1}(z^*) \cap \H_\pm$, $k \in \Z$ such that $|z_k-p_k|< 1$. Therefore, we can take an (almost horizontal) curve $\eta$ in $\H_\pm$, landing at $\infty$ and containing all the points $z_k$ for $k > 0$. Then $f(\eta)$ is a curve that does not land at $\infty$, since it has to go through $z^*$ infinitely many times. This shows that the unique access to $\infty$ is not strongly invariant, giving~(c). 

To prove~(d), it is sufficient to notice that $\bd \H_\pm = \R \cup \{\infty\}$ contains repelling fixed points $v_k = k\pi$ and the poles $p_k$ of $f$ for $k \in \Z$.

Since a Riemann map $\varphi_\pm: \D \to \H_\pm$ is M\"obius, the inner function $g_\pm = \varphi_\pm^{-1} \circ f \circ \varphi_\pm$ has a unique singularity $\varphi_\pm^{-1}(\infty)$, which is the Denjoy--Wolff point for $g_\pm$. This shows~(e). Note also that by the Wolff Boundary Schwarz Lemma, $\Im(f(z)) > \Im (z) $ if $z\in \H_+$ and, by symmetry, $\Im(f(z)) < \Im (z) $ when $z\in \H_-$. 

\end{proof}

\begin{rem*}
Observe that $f(z) = z + \tan z$ provides an example of a meromorphic map with a Baker domain $U$ whose boundary is a Jordan curve in $\clC$ but $f|_U$ is not univalent. This is never the case for entire functions, as shown by Baker and Weinreich in \cite{bw}. Other interesting examples of meromorphic maps with invariant half-planes $\H_\pm$ of the form $f(z) = az +b + \sum_{j= 1}^\infty c_j/(z - d_j)$ with $a, c_j > 0$, $b, d_j \in \R$ are described in \cite{bargmann}. 
\end{rem*}

All the following examples are Newton's methods of finding zeroes of  transcendental entire functions. Recall that for such maps all Fatou components are simply connected (see \cite{bfjk}).

%%%%%%%%%%%%%%%%
\begin{ex} \label{ex5}
Let $f(z)=z-\tan z $, Newton's method of $F(z)=\sin z $. Then the following hold.
\begin{enumerate}[\rm (a)]
\item $f$ has infinitely many immediate basins of attraction $U_k$, $k\in\Z$, of superattracting fixed points $c_k:=k\pi$, such that $\deg f|_{U_k} = 3$. 
\item Each basin $U_k$ has exactly two accesses to infinity, and they are strongly invariant. 
\item $\bd U_k$ contains exactly two accessible poles of $f$.
\item The basins $U_k$ are the only periodic components of $\mathcal F(f)$.
\end{enumerate}
See Figure~\ref{dynplanesine}.
\end{ex}

\begin{figure}[htb!]
\includegraphics[width=0.48\textwidth]{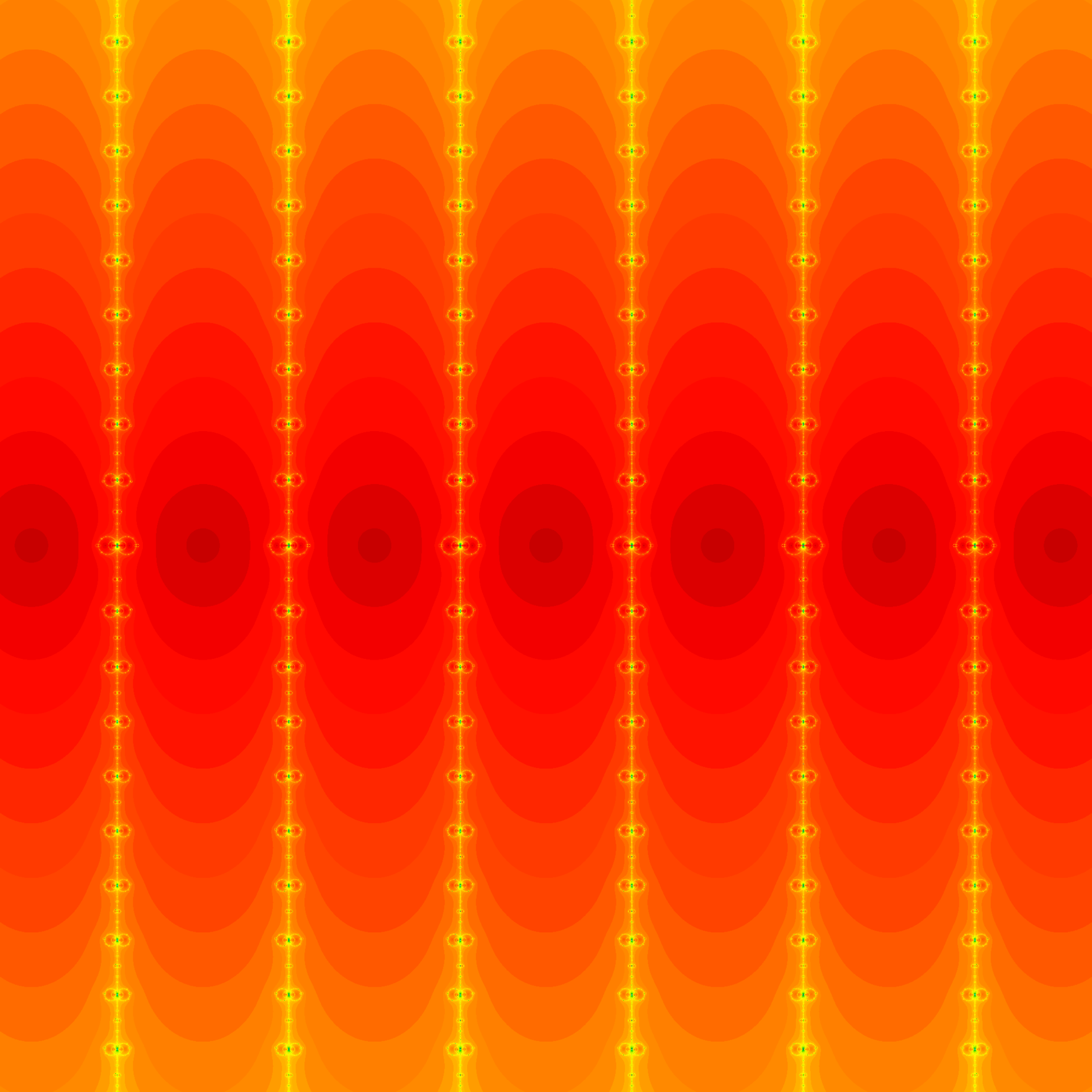}
\hfill
\includegraphics[width=0.48\textwidth]{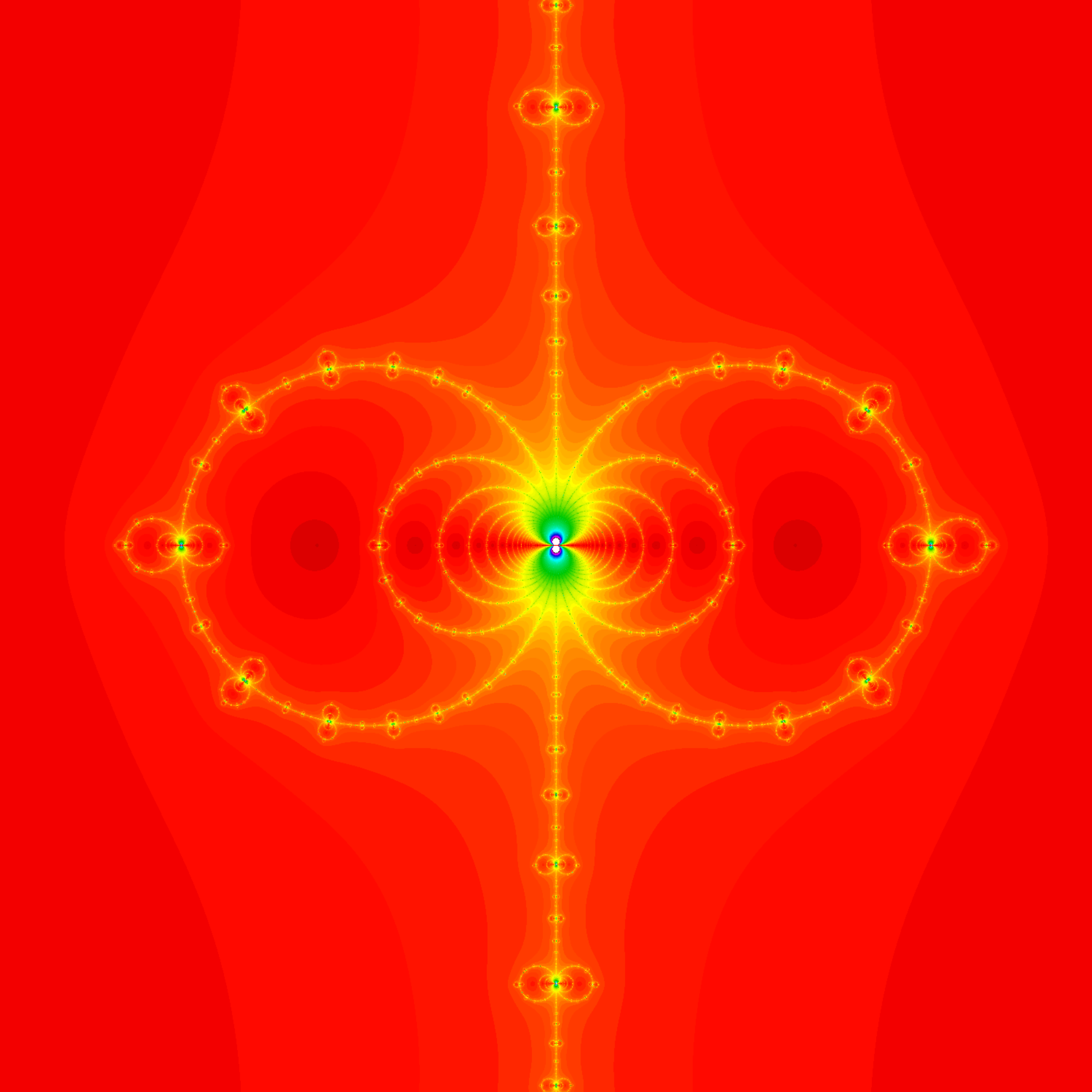}
\caption{\label{dynplanesine} \small Left: The dynamical plane of the map $f$ from Example~\ref{ex5}, showing the invariant attracting basins $U_k$. Right: Zoom of the dynamical plane near a pole $p_k$.}
\end{figure}

\begin{proof}
Note that the points $c_k$, $k \in \Z$ (the zeroes of $\sin z$), are the only fixed points of $f$, they are superattracting and satisfy $f'(c_k)=f''(c_k)=0$, while $f'''(c_k)\neq 0$. Let $U_k$ be the immediate basin of attraction of $c_k$. Since the local degree of $f$ near $c_k$ is $3$, $f|_{U_k}$ has  degree at least $3$. 

It is straightforward to check that the asymptotics of $f$ is as follows:    
\begin{equation}\label{eq:asym}
f(z) \simeq \begin{cases}
z-i & \text{if} \quad \Im (z) \to +\infty \\
z+i & \text{if} \quad \Im (z) \to -\infty.
\end{cases}
\end{equation}
Moreover, the lines
$$r_k(t):=\frac \pi 2 + k \pi + it, \quad \text{for} \quad t \in \R, \; k \in\Z$$
are invariant and contain all the poles
\[
p_k := \pi/2 + k \pi
\]
of $f$. In particular, the above facts imply that $f$ has no finite asymptotic values. 

Now we prove that
\begin{equation}\label{eq:r_k}
r_k \subset \mathcal J(f).
\end{equation}
To see this, note that for $z \in r_k$ we have 
\[
f(z)=f\left(\frac\pi 2 + k \pi+ it\right) = \frac \pi 2 + k \pi+ ig(t),
\]
where
\[
g(t) = t - \coth t.
\]
The function $g$ maps each half-line $\R^+$, $\R^-$ homeomorphically onto $\R$ and
\[
g'(t) = 1 + \frac{1}{\sinh^2 t} > 1
\]
for $t \in \R \setminus \{0\}$ (see Figure~\ref{graphsine}, left).

\begin{figure}[htb!] 
\includegraphics[height=0.35\textwidth]{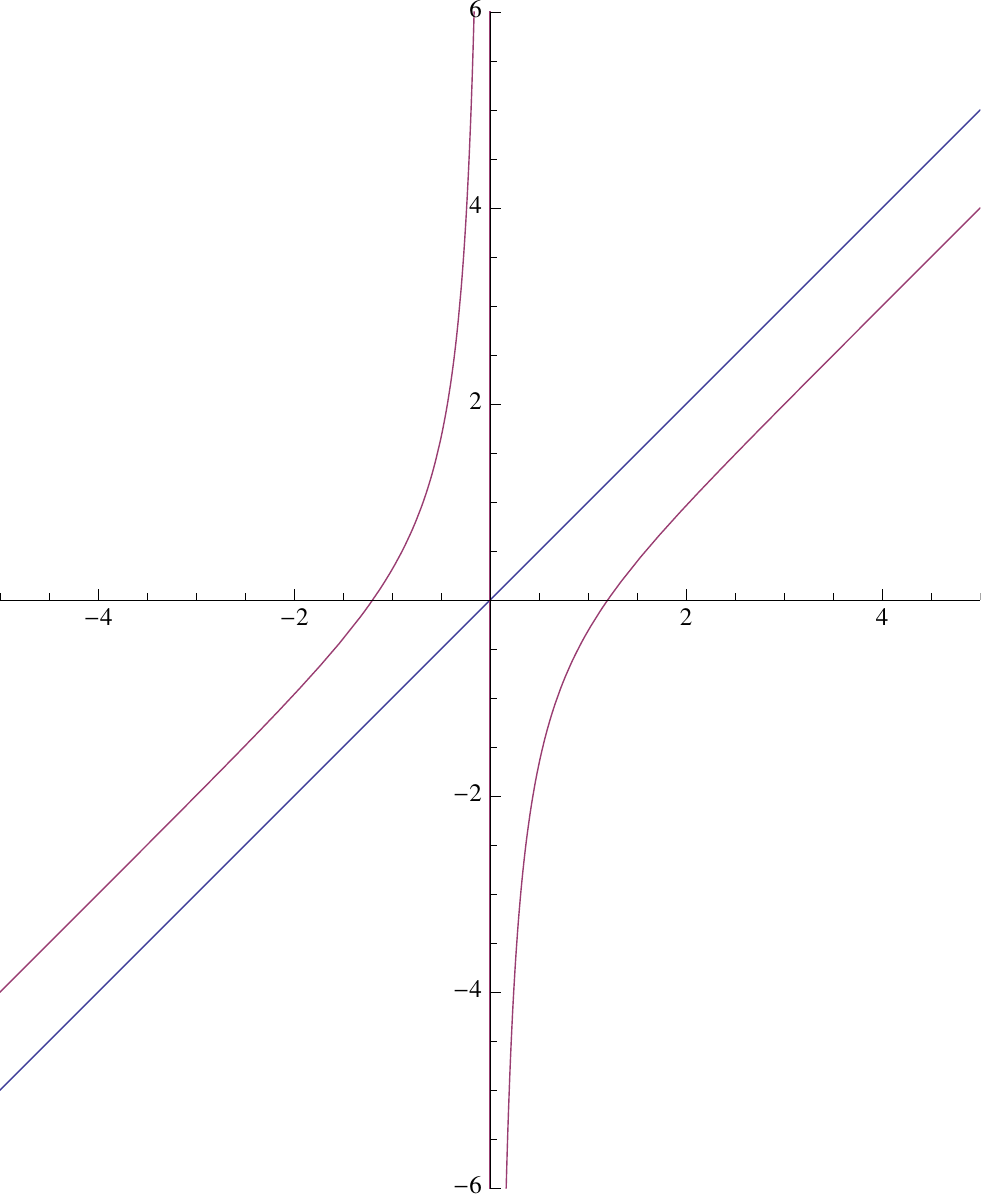}\hfil
\includegraphics[height=0.35\textwidth]{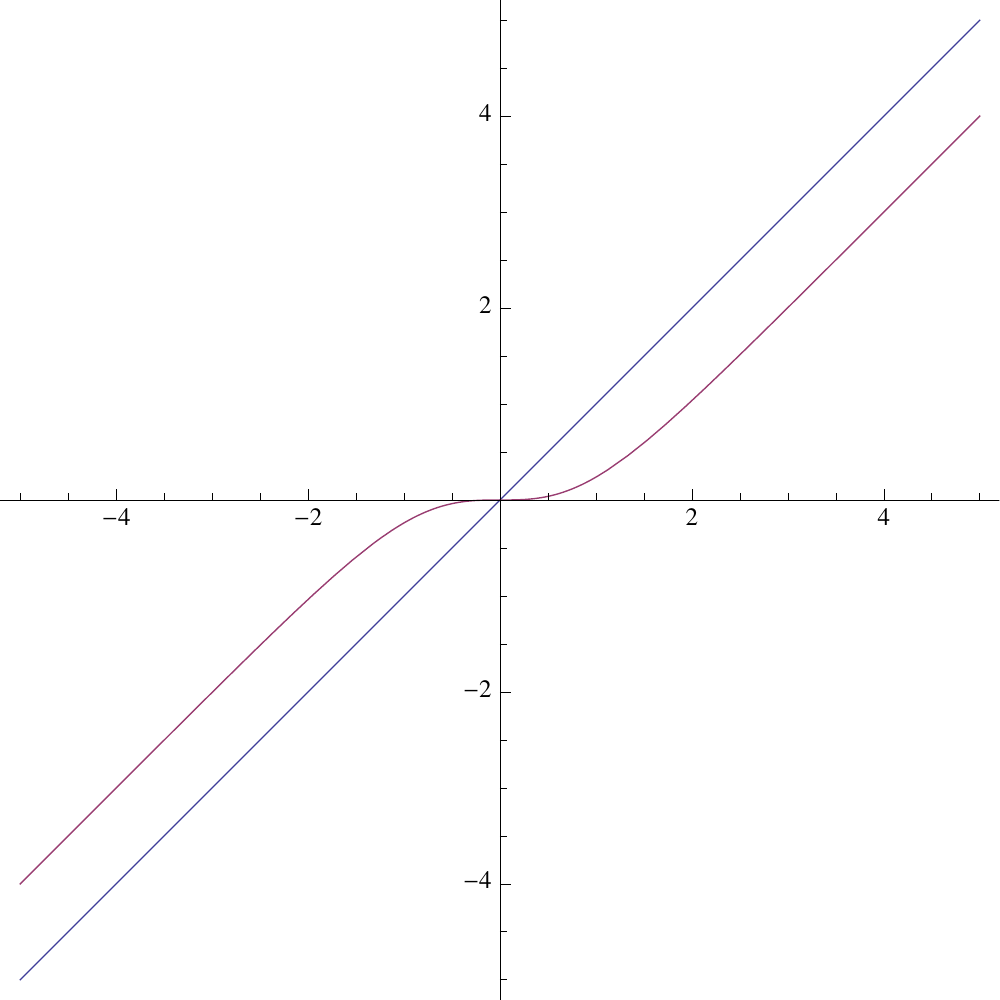}
\caption{\label{graphsine}\small Left: 
The graph of the map $f$ from Example~\ref{ex5} restricted to the vertical line $r_k$, plotting $t\mapsto g(t)=\Im(f(r_k(t)))$. Right: The graph of $f$ restricted to the vertical line $l_k$, plotting $t\mapsto \Im(f(l_k(t)))$.}
\end{figure}

This together with \eqref{eq:asym} implies that for every open interval $I \subset \R \setminus \{0\}$ there exists $n \ge 0$ such that $0 \in g^n(I)$. 
Consequently, for every neighbourhood $W$ of a point $z \in r_k$, the set $f^n(W)$ contains a pole of $f$ for some $n \ge 0$. Since 
\begin{equation}
\label{eq:J}
\mathcal J(f) = \overline{\bigcup_{n = 0}^\infty \bigcup_{k\in\Z} f^{-n}(p_k)},
\end{equation}
this proves \eqref{eq:r_k}. 
It then follows that the basin $U_k$ lies in the vertical strip between $r_{k-1}$ and $r_{k}$. Hence, by \eqref{eq:asym} and Theorem~\ref{degree}, $\deg f|_{U_k}$ is finite. Recall that we showed $\deg f|_{U_k} \geq 3$. The map $f$ has no singular values other than the fixed critical points $c_k$ and no other critical points. Hence, by the Riemann--Hurwitz formula, $\deg f|_{U_k} = 3$. This proves~(a).

To prove~(b), observe that $U_k$ has two (distinct) accesses to $\infty$ defined by two invariant vertical half-lines $l_k^+, l_k^- \subset U_k$, where
\[
l_k^+(t)=\pi k + it, \quad l_k^-(t)=\pi k - it \quad \text{for} \quad t\geq 0,\; k \in \Z.
\]
By \eqref{eq:asym}, if $\gamma$ is a curve in $U_k$ landing at $\infty$, then $f(\gamma)$ also lands at $\infty$, so by Proposition~\ref{strongly}, the two accesses are strongly invariant.

To show that these are the only two accesses to $\infty$ from $U_k$, we will prove that 
\begin{equation} \label{union}
\mathcal{J}(f) \cap \C = \bigcup_{k\in\Z} J_k,
\end{equation}
where $J_k$ is the connected component of $\mathcal{J}(f) \cap \C$ containing the line $r_k$ (and hence the pole $p_k$). To that end, note that 
the post-singular set of $f$ (the closure of the forward trajectories of all singular values of $f$) is equal to $\{c_k \mid k\in\Z\} \cup \{\infty\}$, so all inverse branches of $f^{-n}$, $n > 0$ are defined on $r_k$. This implies that the set $f^{-1}(r_k)$ consists of $r_k$ itself and an infinite number of disjoint topological arcs $A_l$, $l \in \Z \setminus \{k\}$, such that $A_l \cap \bigcup_k l_k = \emptyset$, where
\[
l_k = l_k^+ \cup l_k^- \subset U_k,
\]
and both ends of $A_l$ land at the pole $p_l$.  
It follows that the set 
\[
R^n :=f^{-n}\left(\bigcup_{k\in\Z} r_k \cup \{\infty\}\right), \quad n > 0
\]
consists of connected components $R^n_k$, $k \in \Z$, such that $R^n_k$ lies in the vertical strip between the lines $l_k$ and $l_{k+1}$, and
\[
R^1_k \subset R^2_k \subset  R^3_k \subset\cdots.
\]
Consequently, for each $k\in\Z$, the set $\overline{\bigcup_{n=1}^\infty R^n_k}$ contains $r_k$ and is connected. By \eqref{eq:r_k} and \eqref{eq:J}, we obtain
\[
\mathcal J(f) \cap \C = \overline{\bigcup_{n=1}^\infty R^n} = \bigcup_{k\in\Z}\overline{\bigcup_{n=1}^\infty R^n_k},
\]
which shows that $J_k = \overline{\bigcup_{n=1}^\infty R^n_k}$ and proves \eqref{union}. Now \eqref{union} and Lemma~\ref{lem:no_boundary1} show that $U_k$ has only two accesses to $\infty$, which ends the proof of (b).

To show~(c), it is enough to check that the preimages of the lines $l_k^\pm$ (respectively $l_{k+1}^\pm$) under suitable inverse branches of $f^{-1}$ are curves in $U_k$ landing at the pole $p_{k-1}$ (respectively $p_k$). This implies that the poles $p_{k-1}$, $p_k$ are accessible from $U_k$.

To prove~(d), recall that $f$ has no singular values apart from the critical points $c_k$, hence any periodic Fatou component is either one of the superattracting basins $U_k$ or a Baker domain. But the latter is not possible for if $U$ were a Baker domain of $f$, then curves in its dynamical access to infinity would have unbounded imaginary part, which is prevented by \eqref{eq:asym}.

\end{proof}

%%%%%%
\begin{ex} \label{ex2} 
Let $f(z)=z+i + \tan z$, Newton's method of  $F(z)=\exp \left(-\int_0^z \frac{du}{i+\tan u}\right)$.
Then we show the following properties. 
\begin{enumerate}[\rm (a)]
\item $f$ has a completely invariant Baker domain $U$ (in particular $\deg f|_U = \infty$), which contains the upper half-plane $\H_+$ and the vertical lines $l_k:=\{\Re (z)=k\pi\}$ for all $k\in\Z$.
\item The Baker domain $U$ has infinitely many strongly invariant accesses to infinity and the dynamical access to infinity from $U$ is invariant but not strongly invariant.

\item $\partial U$ contains infinitely many accessible poles of $f$.

\item The inner function associated to $f|_U$ has a unique singularity in $\bdd$, so $f|_U$ is singularly nice.

\item The map $f$ has infinitely many invariant Baker domains $U_k$, $k\in\Z$, such that $U_k$ contains the vertical half-lines $s_k^-:=\{ \Re (z) = \frac{\pi}{2} + k \pi, \; \Im (z) <0 \}$ and $\deg f|_{U_k} = 2$.

\item The Baker domain $U_k$ has exactly one access to infinity, and it is strongly invariant. 

\item $\bd U_k$ contains exactly one accessible pole of $f$.

\end{enumerate}
\end{ex}

See Figure~\ref{dynplane2}.
\begin{figure}[htb!]
\includegraphics[width=0.7\textwidth]{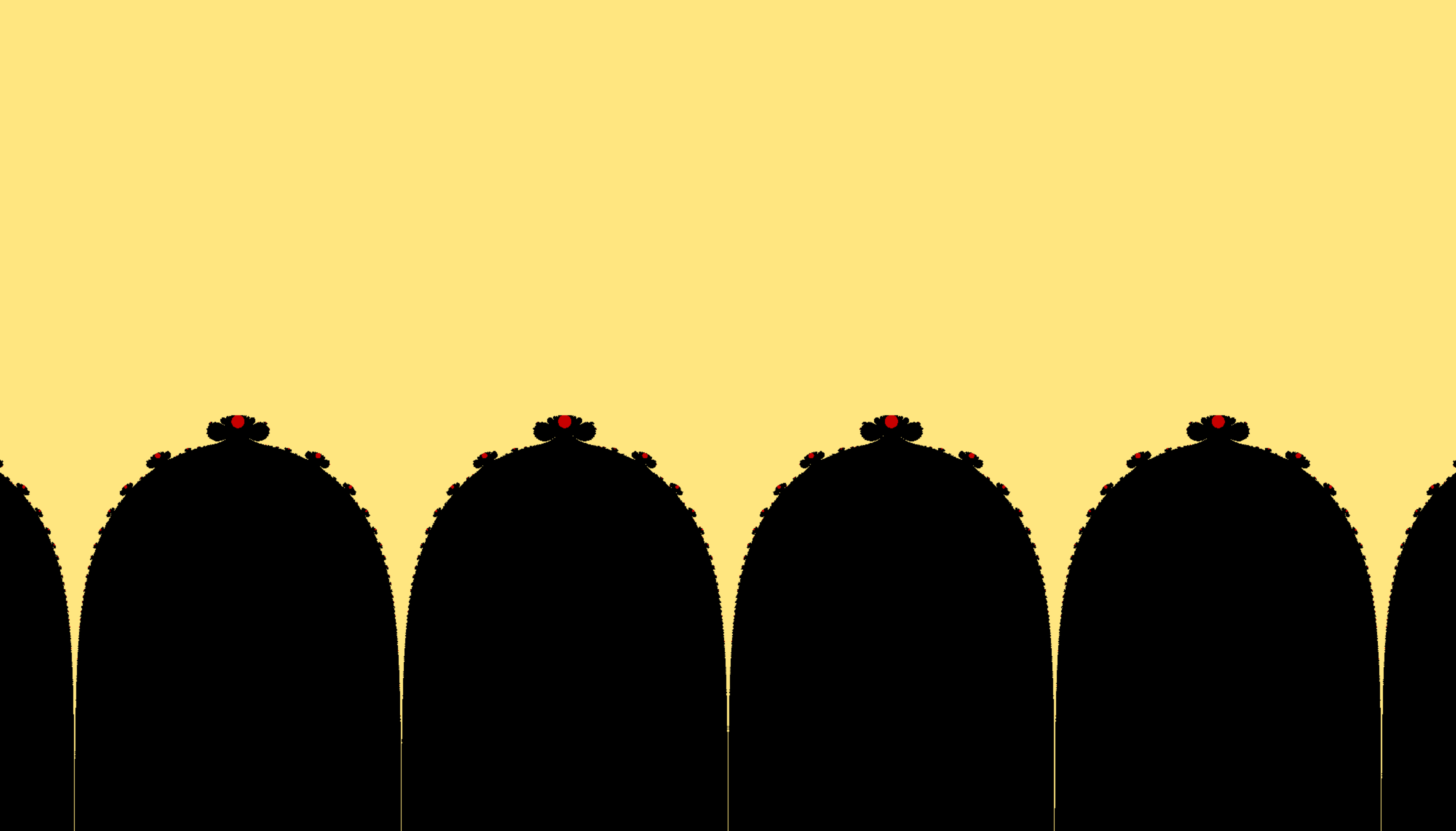}
\caption{\label{dynplane2} \small The dynamical plane of $f$ from Example~\ref{ex2}, showing the invariant Baker domain $U$ of infinite degree (yellow) and invariant Baker domains $U_k$ of degree~$2$ (black).
%Range $[-7,7]\times [-4,4]$.
}
\end{figure}

\begin{proof}
Observe that $f$ has no fixed points, since $z=-i$ is an omitted value for the tangent map. Note that by \eqref{eq:tan}, we have asymptotically 
\begin{equation}\label{eq:asympt2}
f(z) \simeq \begin{cases}
z+2i & \text{if} \quad \Im (z) \to +\infty \\
z& \text{if} \quad \Im (z) \to -\infty.
\end{cases}
\end{equation}
Also from \eqref{eq:tan}, it follows that the lines $l_k$ are forward invariant and $\Im (f(z))  > \Im (z)$ for $z\in \H_+ \cup L$, where $L=\bigcup_{k\in\Z} l_k$. Hence, the set $\H_+ \cup L $ is forward invariant and $\Im(f^n(z)) \to +\infty$ as $n \to \infty$ for $z \in \H_+ \cup L$. Therefore, $\H_+ \cup \mathcal{L}$ is contained in an invariant Baker domain $U$. The real line is mapped infinitely-to-one to the horizontal line $l:=\{\Im (z) =1\}$ together with $\infty$, which shows that $\deg f|_U = \infty$. Moreover, there are no other preimages of $l$, hence $U$ is completely invariant. This finishes the proof of~(a).

The dynamical access to $\infty$ from $U$ is the equivalence class of the invariant curve $\gamma(t) = it$, $t \ge 0$, so it is invariant. Arguing as in Example~\ref{ex1}, we construct a curve $\eta$ in this access, such that $\eta$ contains infinitely many preimages under $f$ of a given point $z^*$ with a large imaginary part (they are located nearby the poles). Then it is clear that $f(\eta)$ cannot land at infinity, and hence the dynamical access to $\infty$ is not strongly invariant. In fact, there are two curves $\eta_1, \eta_2$ in the dynamical access to infinity from $U$, such that its union contains the whole set $f^{-1}(z^*)$. Considering a Riemann map $\varphi:\D\to U$ and the inner function $g=\varphi^{-1} \circ f \circ \varphi$ associated to $f|_U$ and using the Correspondence Theorem, we see that the curves $\varphi^{-1}(\eta_1)$, $\varphi^{-1}(\eta_2)$ land at a common point in $\partial \D$ and $\varphi^{-1}(\eta_1) \cup \varphi^{-1}(\eta_2)$ contains all preimages under $g^{-1}$ of a given point in $\D$. We conclude from Theorem~\ref{thm:bargmann}~part~(c) that $g$ has exactly one singularity, so $f|_U$ is singularly nice. By Corollary~C, the Baker domain $U$ has infinitely many invariant accesses to infinity. Observe that these accesses correspond to curves homotopic to $l_k^-$, where
\[
l_k^-(t) = k\pi -it \quad \text{for} \quad t \ge 0,
\]
and they lie between $s_k^-$ and $s_{k+1}^-$. By \eqref{eq:asympt2}, the image of any curve landing at infinity in such access also lands at infinity. It follows from Proposition~\ref{strongly} that these accesses are strongly invariant. This proves~(b) and~(d). 

Since the poles 
$$p_k := \frac{\pi}{2} + k\pi, \quad k \in \Z$$
of $f$ are the endpoints of $s_k^-$, they are accessible from $U$, which shows~(c).

To show~(e), notice that the vertical lines 
\[
s_k = \left\{\Re (z) =\frac \pi 2+k\pi\right\}
\]
are invariant under $f$. Each of them contains the pole $p_k$ (at its intersection with $\R$) and two critical points, symmetric with respect to the real line. By \eqref{eq:tan}, on the half-lines $s_k^-$ we have $\Im (f(z)) <  \Im (z)$. Since $f$ has no fixed points, this implies that $\Im(f^n(z))$ tends to $-\infty$ as $n \to \infty$ for $z\in s_k^-$. Moreover, $s_k^- \subset \mathcal{F}(f)$. To see this, consider the half-strip
\[
S_k = \left\{\left|\Re (z) - \frac \pi 2 - k\pi\right| < \frac \pi 8, \; \Im(z) < -\frac {\ln 2} 2 \right\}. 
\]
Using \eqref{eq:tan}, one may check directly that $f(\overline{S_k}) \subset S_k$, which implies that $s_k^- \cup S_k$ is contained in an invariant Fatou component of $f$. Since the points in this set escape to infinity under iteration, this Fatou component is an invariant Baker domain which we denote by $U_k$. The domains $U_k$, $k \in \Z$ are disjoint, because they are separated by the lines $l_k \subset U$. Moreover, each half-line $s_k^-$ contains a critical point $c_k$. We conclude that $f$ has infinitely many Baker domains $U_k, {k\in\Z}$, such that $U_k$ contains the half-line $s_k^-$, the critical point $c_k$ 
and the unique accessible pole $p_k$ in its boundary. In particular, this shows~(g).

By \eqref{eq:asympt2} and the fact that every $U_k$ contains exactly one pole of $f$ on its boundary, it is clear that the map $f|_{U_k}$ has finite degree and hence $\deg f|_{U_k} \in \{1,2\}$ by Theorem~D. Since $U_k$ contains a critical point, in fact $\deg f|_{U_k} = 2$. This ends the proof of~(e).

To show~(f), note that by Theorem~D, each $U_k$ has a unique access to infinity (the dynamical access containing $s_k^-$) and it is invariant. In fact, this access is strongly invariant by Proposition~\ref{strongly} and the asymptotic estimates \eqref{eq:asympt2}. 
\end{proof}

%%%%%%%%%%%%%%%%
\begin{ex} \label{ex3} 
Let $f(z)=\frac{e^z(z-1)}{e^z+1}$, Newton's method of $F(z)=z+e^z$. Then  the following hold. 
\begin{enumerate}[\rm (a)]
\item $f$ has a completely invariant immediate basin of attraction $U_0$ of a superattracting fixed point $c_0\in \R^-$ (in particular $\deg f|_{U_0} = \infty$) such that $U_0$ contains the real line and the region 
$$R = \{x+iy \mid |y| \le c e^{-x},\; x < x_0\},$$
for some $c > 0$, $x_0 < 0$.

\item $U_0$ has infinitely many invariant accesses to infinity.

\item The inner function associated to $f|_{U_0}$ has a unique singularity, so $f|_{U_0}$ is singularly nice. 

\item $f$ has infinitely many basins of attraction $U_k$, $k \in \Z\setminus \{0\}$ of superattracting fixed points $c_k \in \C \setminus \R$, such that $U_k$ has at most one access to infinity.
 
\end{enumerate}

See Figure~\ref{dynplane3}.

\begin{figure}[htb!]
\includegraphics[width=0.48\textwidth]{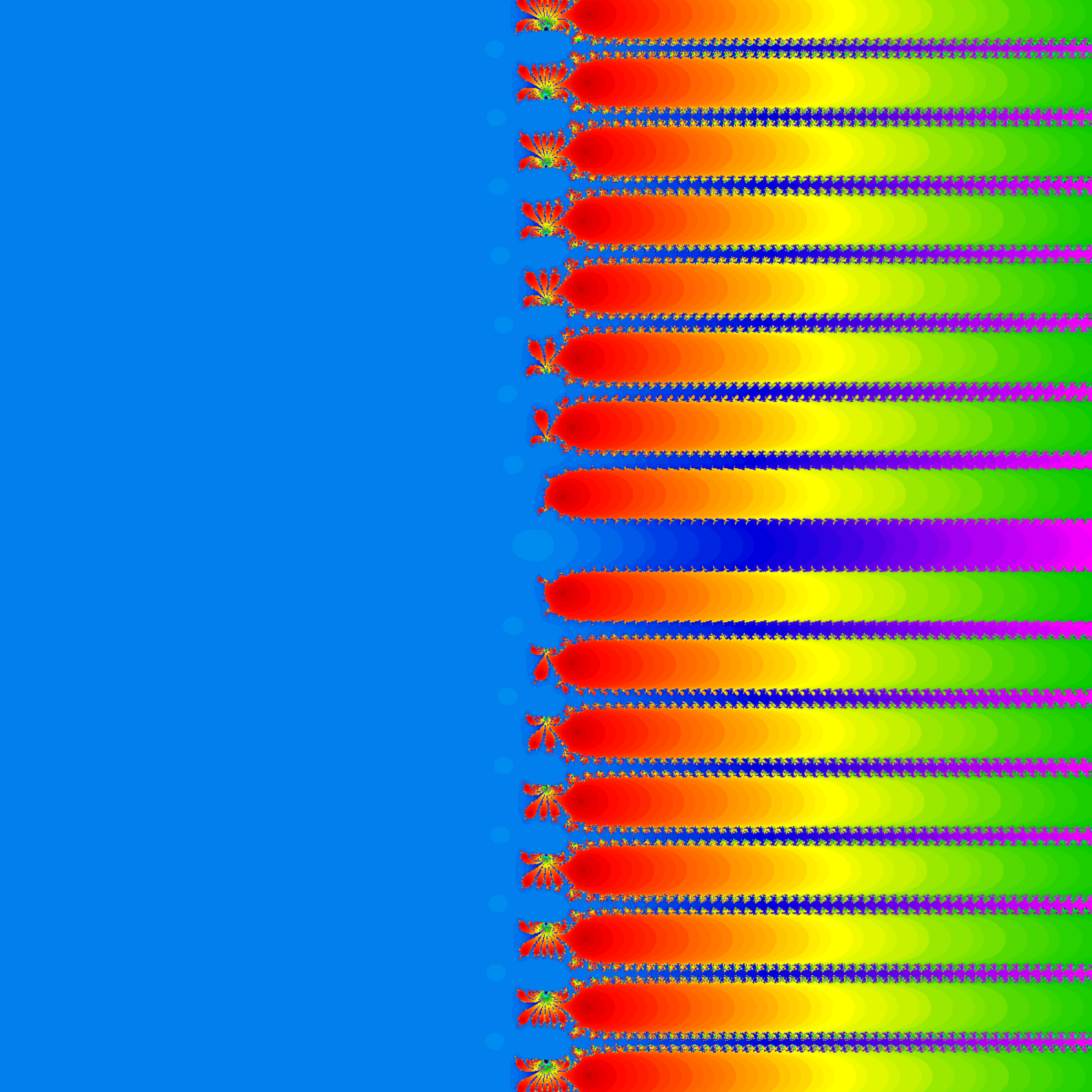}
\hfill
\includegraphics[width=0.48\textwidth]{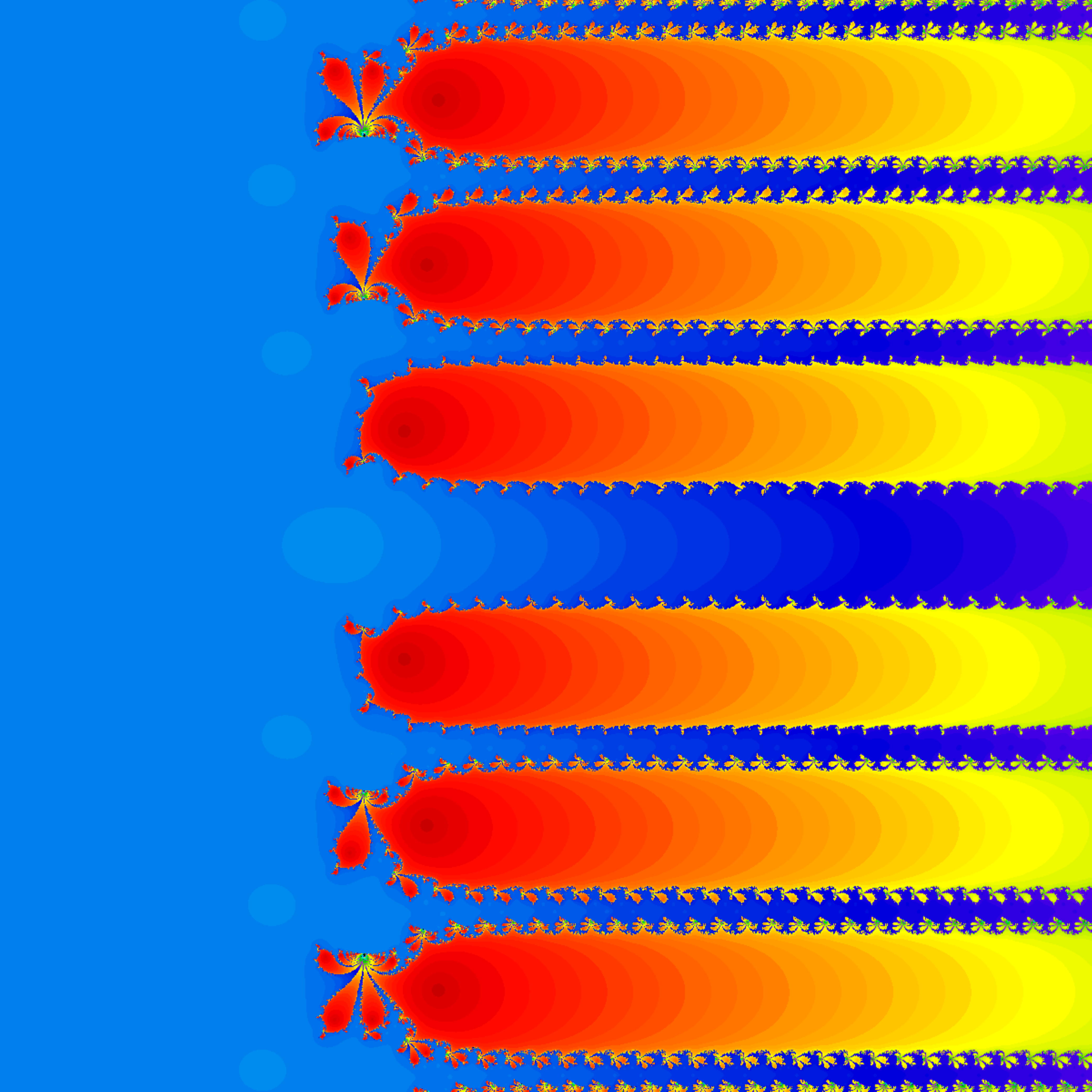}
\caption{\label{dynplane3} \small Left: The dynamical plane of the map $f$ from Example~\ref{ex3}, showing the invariant attracting basin $U_0$ (blue) and other attracting basins $U_k$. Right: Zoom of the dynamical plane. 
}
\end{figure}

\begin{proof}
We start by observing that the real line is invariant under $f$. The graph of $f|_\R$ is presented in Figure~\ref{graph3}.
\begin{figure}[htb!]
\includegraphics[width=0.35\textwidth]{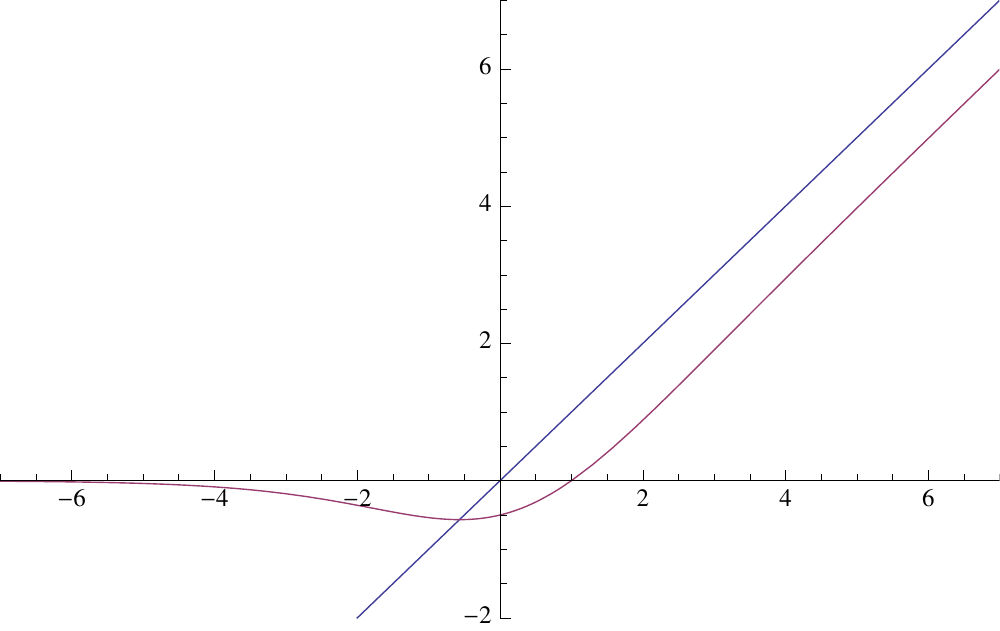}
\caption{\label{graph3} \small The graph of the map $f$ from Example~\ref{ex3} restricted to the real line.}
\end{figure}
In particular, the equation $F(z) = z + e^z = 0$ has a unique (simple) real zero $c_0$.  Moreover, there are infinitely many non-real (simple) zeroes of this equation, which we denote by $c_k$, $k \in \Z \setminus \{0\}$. Observe that $c_k$ are superattracting fixed points of $f$, and there are no other ``free'' critical points of $f$. Denote by $U_k$, $k \in \Z$ the immediate basins of attraction of the points $c_k$.

Every point in $\R$ converges to $c_0$ under the iteration of $f$, including the (unique) asymptotic value $v=0$ of $f$ and its unique preimage $u=1$. The asymptotic value $v$ has a unique asymptotic tract. It is straightforward to check that the region $R$ is mapped by $f$ into a small neighbourhood of $0$ (provided $c$ is sufficiently small and $x_0$ has sufficiently large negative part), so it is contained in $U_0$.
On the other hand, for a point $w$ in a small neighbourhood of $0$, the set $f^{-1}(w)$ consists of infinitely many points in $R \subset U_0$ and exactly one point in a small disc around $u=1$, which is contained in $U_0$. This implies that the basin $U_0$ is completely invariant, in particular $\deg f|_{U_0} = \infty$. This ends the proof of~(a).

Consider a Riemann map $\varphi: \D \to U_0$ and the inner function $g = \varphi^{-1}\circ f \circ \varphi: \D \to \D$ associated to $f|_{U_0}$. 
Notice that the half-line $R \cap \R^-$ defines an access $\A$ to infinity from $U_0$. Moreover, every sequence of points in $R$ converging to $\infty$ lies on a curve $\gamma$ in $R$ landing at $\infty$, such that $\gamma \in \A$. Hence, by the Correspondence Theorem, $\overline{\varphi^{-1}(R)} \cap \bdd$ is equal to a unique point $\zeta \in \bdd$ corresponding to the access $\A$. Since $R$ contains all but one preimages of a given point $w \in U_0$ sufficiently close to $0$, by Theorem~\ref{thm:bargmann}~part~(c), the point $\zeta$ is the unique singularity of $g$. It follows that $f_{U_0}$ is singularly nice so, by Corollary~C, the basin $U_0$ has infinitely many invariant accesses to infinity, which proves~(b) and~(c).  

Since $U_0$ is completely invariant, it follows from Theorem~D that for every $k\in\Z\setminus\{0\}$, the attracting basin $U_k$ has at most one access to infinity, which shows~(d). Since $U_k$ contains a critical point of $f$, Theorem~D implies also that $\deg f|_{U_k}$ is equal to $2$ or $\infty$.

\end{proof}

\begin{rem*}
Computer pictures suggest that the accesses to infinity in the basins $U_k$, $k \in \Z \setminus \{0\}$ in the right half-plane are asymptotically horizontal. If this were true, given that $f(z)\sim z-1$ when $\Re (z) \to \infty$, it would imply that all of them are strongly invariant. Indeed, in this case any curve tending to infinity with real part tending to $+\infty$, satisfies that its image also tends to $\infty$. Hence, by Proposition~\ref{strongly}, the corresponding access is strongly invariant. Moreover, the degree of $f$ on $U_k$ would be finite, so by Theorem~D, $\deg f|_{U_k} = 2$ and $\bd U$ contains exactly one pole of $f$ accessible from $U_k$.
\end{rem*}

\end{ex}

\bibliography{accesses}

\end{document}